\documentclass{article}
\usepackage[utf8]{inputenc}
\usepackage[dvipsnames]{xcolor}
\usepackage{amsthm}
\usepackage{amsmath}
\usepackage{enumitem}
\usepackage{algorithm}
\usepackage{physics}
\usepackage{amssymb}
\usepackage{algpseudocode}
\usepackage{mathtools}
\usepackage[T1]{fontenc}
\usepackage{amsfonts}
\usepackage{hyperref}
\usepackage{enumitem} 
\usepackage{caption}
\usepackage{mathrsfs}
\usepackage{float}
\usepackage{array}
\usepackage{subcaption}
\usepackage{makecell}
\usepackage{booktabs}
\usepackage{siunitx}
\usepackage[numbers,sort&compress]{natbib}
\bibpunct[, ]{[}{]}{,}{n}{,}{,}

\usepackage{todonotes}

\usepackage{braket}
\usepackage[a4paper, left=3.5cm, right=3.5cm, top=3cm]{geometry}
\newtheorem{theorem}{Theorem}[section]
\newtheorem{definition}[theorem]{Definition}
\newtheorem{lemma}[theorem]{Lemma}
\newtheorem{assumption}[theorem]{Assumption}

\newtheorem{corollary}[theorem]{Corollary}
\newtheorem{remark}[theorem]{Remark}

\newcommand{\R}{\mathbb{R}}
\newcommand{\N}{\mathbb{N}}
\newcommand{\Z}{\mathbb{Z}}

\newcommand{\Linf}{L^{\infty}(\Omega)}

\newcommand{\io}{\int_{\Omega}}

\newcommand{\Ltwo}{L^2(\Omega)}
\newcommand{\Hone}{H_0^1(\Omega)}
\newcommand{\dom}{\operatorname{dom}}
\newcommand{\sign}{\operatorname{sign}}

\newcommand{\proj}{\operatorname{proj}}

\newcommand{\prox}{\operatorname{prox}}
\newcommand{\ared}{\operatorname{ared}}
\newcommand{\pred}{\operatorname{pred}}
\newcommand{\cred}{\operatorname{cred}}

\newcommand{\Uad}{U_{ad}}
\newcommand{\Iuad}{I_{\Uad}}
\newcommand{\kp}{_{k+1}}

\begin{document}

	\title{A smoothed proximal trust-region algorithm for nonconvex optimization problems with $L^p$-regularization, $p\in (0,1)$ 
    \thanks{This work is partially supported by the Office of Naval Research (ONR) under Award NO: N00014-24-1-2147, NSF grant DMS-2408877, the Air Force Office of Scientific Research (AFOSR) under Award NO: FA9550-22-1-0248 and the German Research Foundation DFG under project grant Wa 3626/5-1.}
    \author{Harbir Antil\thanks{Department of Mathematical Sciences and Center for Mathematics and Artificial Intelligence (CMAI), George Mason University 
  ({hantil@gmu.edu}).}
\and Anna Lentz\thanks{Institut für Mathematik, Universität Würzburg
  ({anna.lentz@uni-wuerzburg.de}).}
}
}
	\maketitle

	\begin{abstract}
		We investigate a trust-region algorithm to solve a nonconvex optimization problem with $L^p$-regularization for $p\in(0,1)$. The algorithm relies on descent properties of a so-called generalized Cauchy point that can be obtained efficiently by a line search along a suitable proximal path. To handle the nonconvexity and nonsmoothness of the $L^p$-pseudonorm, we replace it by a smooth approximation and construct a convex upper bound of that approximation. This enables us to use results of a trust-region method for composite problems with a convex nonsmooth term.
		We prove convergence properties of the resulting smoothed proximal trust-region algorithm and investigate its performance in some numerical examples. Furthermore, approximate subproblem solvers for the arising trust-region subproblems are considered.
	\end{abstract}
	
	\section{Introduction}
	We consider the optimization problem 
	\begin{equation}\label{eq:minprobLp}
		\min_{u\in \Uad}  f(u) + \io |u|^p \dd x \qquad \text{for } p\in (0,1), \qquad \Uad \subseteq U,
	\end{equation}
	where $U$ represents either the space $\Ltwo$ or $\Hone$ for some bounded domain $\Omega\subset\R^d$, 
	$f:U\to \R$ is a Fréchet differentiable function with Lipschitz continuous gradient and $\Uad$ represents some admissible set.
	The nonconvex $L^p$-pseudonorm promotes sparsity in solutions, making this problem relevant in various applications. 
	However, the nonconvexity and nonsmoothness of the objective pose 
	theoretical and computational challenges. 
	The aim of this paper is to develop a trust-region algorithm that solves this problem numerically. \\
	Problems of the above form have already been studied thoroughly in the literature, for example in the context of inverse problems \cite{inverseLqRobin331} or sparse optimal control. In the latter case, the $L^p$-pseudonorm leads to sparse controls, which is for example desirable in actuator placement problems. The article \cite{lp_cont} investigates such optimal control problems, derives necessary optimality conditions and discusses existence of solutions and regularization methods. 
	Based on this regularization, \cite{sparse} develops necessary optimality conditions in the (fractional) Sobolev setting. \\ 
	Numerically solving such problems can be challenging due to the nonconvexity of the $L^p$-term.
	Existing approaches include proximal gradient methods \cite{15proxNonsmooth}, majorize-minorization methods with a smoothing scheme \cite{sparse, lp_cont} and difference-of-convex methods \cite{87huberApprox}. \\
	More general composite problems
	\[
	\min_{u\in U} f(u) + \phi(u)
	\]
	consisting of the sum of a (nonconvex) smooth function $f$ and a nonconvex and nonsmooth function $\phi$ can be solved with proximal gradient methods \cite{261convDescent}, other splitting methods \cite{DouglasRachfordNonconvDC325} or alternating minimization algorithms, often under the Kurdyka–Łojasiewicz condition \cite{264proxAlternating, proxAltMinNonconv332}.\\ 
	However, these methods are first order methods and thus possibly slowly converging. Therefore, improvements like an accelerated proximal gradient method were considered in \cite{260accProxGrad}.
	Second-order methods, such as proximal quasi-Newton (see e.g. \cite{globInexProxNewton333}) and trust-region algorithms offer faster convergence but traditionally require convexity.\\ 
	Such a trust-region method for composite problems with the nonsmooth part being convex was introduced in \cite{240proxTrustRegion}.  This trust-region algorithm uses a so-called generalized Cauchy point to compute a trial step that satisfies a descent condition which is sufficient for convergence of the algorithm, analogous to a Cauchy point in the standard trust-region algorithm. This generalized Cauchy point (GCP) can be computed efficiently using the proximal operator. This fact distinguishes the algorithm from other approaches as for example in \cite{275proxQuasiNewtonTr}. While that method also applies to nonconvex nonsmooth functions, it requires the computation of a proximal point of $\phi$ plus the indicator function of the trust-region, which is more difficult. \\
	However, extending the GCP approach from \cite{240proxTrustRegion} to nonconvex problems is non-trivial. The nonconvexity breaks key properties of the proximity operator, making it difficult to ensure descent conditions or global convergence. \\
	In this paper, we aim to generalize the trust-region algorithm from \cite{240proxTrustRegion} from the case of a convex function $\phi$ to the nonconvex $L^p$-pseudonorm.
	Due to the nonconvexity we were not able to prove that nonconvex generalized Cauchy points satisfy a descent condition that leads to convergence of the algorithm as in the convex case. At each iteration, we therefore construct a convex upper bound of the $L^p$-term with the help of a smooth approximation. This allows us to leverage some of the existing convergence theory from the convex case to the nonconvex $L^p$-case, but several new results are required.
	Thus, we first smooth and convexify the $L^p$-pseudonorm and then compute proximal points of the convex surrogate function. Another approach that we do not pursue further could be to keep working with the $L^p$-pseudonorm, but use an approximation of the proximal point instead. \\
	Moreover, \cite{240proxTrustRegion} allows for inexact objective function and derivative evaluations. This is relevant when $f$ or its gradient can not be evaluated efficiently as for example in PDE-constrained optimization. Therefore, we have also generalized the handling of inexactness from \cite{240proxTrustRegion} to our $L^p$-setting.\\
	The performance of trust-region algorithms depends significantly on the chosen subproblem solver. In \cite{240proxTrustRegion}, a spectral proximal gradient (SPG) method was used. Other subproblem solvers or extensions of the SPG method function were considered in \cite{284effProxSubproblTr, 282projProxGradTr}. However, those solvers make use of the convexity of the nonsmooth part in the composite problem, so they can not be directly transferred to the $L^p$-setting. Therefore, we use some kind of spectral proximal gradient method in our implementation, or a majorize-minimization (MM) algorithm based on the scheme in \cite{sparse}. The latter MM scheme does not involve proximal point computations of the $L^p$-pseudonorm. This also allows for the application in the $\Hone$-setting, where the proximal point computation can not be reduced to pointwise optimization problems as in the $\Ltwo$-case. \\
	The structure of the paper is as follows. Section \ref{sec:prelim} provides preliminaries by investigating properties of the proximal map in the nonconvex setting, deriving necessary optimality conditions for \eqref{eq:minprob} and introducing a smooth approximation of the $L^p$-pseudonorm. Based on this, section \ref{sec:algo} presents a trust-region algorithm  and studies its convergence properties. Section \ref{sec:GCP} analyzes the use of generalized Cauchy points in the trust-region algorithm to guarantee sufficient decrease of the iterates. The subsequent section \ref{sec:subproblem} examines numerical algorithms to solve the subproblems in each iteration of the trust-region algorithm. Finally, section \ref{sec:numRes} presents some numerical results.
		
	\section{Preliminaries}\label{sec:prelim}
	\subsection{Notation and standing assumptions}
	Throughout this paper, we adopt the Hilbert spaces $\Ltwo$ and $\Hone$ as the underlying function spaces. To refer to both of these spaces, we use the notation $U$ with norm $\|\cdot\|$ and inner product $(\cdot, \cdot)$. If we want to specify which of these two spaces is considered, we use the subscripts $\Ltwo$ or $\Hone$ whenever this is not clear from the setting of the respective section. We often identify the dual $U^*$ with $U$ via the Riesz representation theorem. \\
	Note that we could also work with fractional Sobolev spaces $H^s(\Omega)$ for $s\in (0,1]$ as in \cite{sparse}, but for simplicity we restrict ourselves to the non-fractional $\Hone$-case in this paper. \\
	The $L^p$-pseudonorm for $p\in (0,1)$ is denoted by
	\[
	j(u) \coloneqq \|u\|_p^p =\io |u|^p \dd x.
	\]
	It is called pseudonorm since it is concave and therefore not a norm.\\
    We rewrite problem \eqref{eq:minprobLp} as 
    \begin{equation}\label{eq:minprob}
        \min_{u\in U}  f(u) + \phi(u) 
    \end{equation}
    and impose the following assumptions.
	\begin{assumption}\label{ass:fPhi}
		\begin{enumerate}
			\item Let $\Omega\subset\R^d$ for $d\ge 1$ be a bounded domain.
			\item \label{ass:fPhifPart} $f:U\to \R$ is $L$-smooth, i.e. there exists an open set $W\subseteq U$ with $\dom \phi \subseteq W$ such that $f$ is Fréchet differentiable on $W$ and its gradient is Lipschitz continuous with Lipschitz constant $L>0$.
			\item  Let $\phi: U \to \bar\R$ be defined as
			\begin{equation}\label{eq:defPhi}
				\phi(u) = \beta \io |u|^p \dd x + \frac{\alpha}{2}\|u\|^2 + \Iuad(u)
			\end{equation}
			with $\alpha\ge 0, \beta>0$ and $p\in (0,1)$. For $U=\Ltwo$, the admissible set $\Uad$ is defined by
			\begin{equation}\label{def:Uad}
				\Uad \coloneq \{ u \in \Ltwo: a \le u \le b \}, 
			\end{equation}
			where $a, b\in \Ltwo$ such that $a< b$. \\
			For $U=\Hone$, we only consider the unconstrained case with $\Uad=\Hone$ and assume $\alpha>0$.
			\item The objective function $F\coloneq f + \phi$ is bounded below. 
		\end{enumerate}
	\end{assumption}
	For the convergence analysis later on we will impose some assumptions on the curvature of $f$. 
	We use the same notion of curvature as in \cite{240proxTrustRegion, 241globalConvTrustRegion}: For a Fréchet differentiable functional $h: U \to \R$, the curvature of $u\in U$ in direction $s\in U$ is defined as 
	\begin{equation}\label{eq:defCurvature}
		\omega(h,u,s) \coloneqq \frac{2}{\|s\|^2}(h(u+s)-h(u)-\braket{\nabla h(u),s}).
	\end{equation}
	As discussed in \cite{240proxTrustRegion}, the curvature  $\omega(f,u,s)$ of $f$ satisfying assumption \ref{ass:fPhi} is bounded by $L$ for $u, u+s \in W$.
	Furthermore, for quadratic $h$ with self-adjoint Hessian operator $B$ it holds 
	\[
	\omega(h,u,s) = \frac{\braket{s,Bs}}{\|s\|^2}.
	\]
	
		\subsection{Nonconvex proximal points}
		The proximal operator of a function $\tilde\phi:U \to \R$ is defined for some $r>0$ by	
		\[
		\prox_{r \tilde\phi}(u) = \arg\min_{v\in U}  \frac{1}{2r}\|v-u\|^2 + \tilde\phi(v).
		\]
		For convex functions, this map is known to be single valued (\cite[Chapter 12]{267convAnaMonOp}). However, this does not hold in general in the nonconvex case, see Remark \ref{rm:proxNotSingleVal} later on. In this paper, we also use the notation $\prox_{r\tilde\phi}(u) $ to refer to elements of $\prox_{r\tilde\phi}(u)$.\\
		Moreover, several other favorable properties from the convex case do not extend to the nonconvex regime. In view of this, we begin by examining which results from \cite{240proxTrustRegion} remain valid when the convexity assumption is dropped. \\
	We start by generalizing \cite[Lemma 1.1]{240proxTrustRegion}.
	
	\begin{lemma}\label{lm:lm1proxIneq}
		Let  $\tilde\phi: U\to \bar \R$ be some proper and lower semicontinuous function. For the proximal operator of $\tilde\phi$ it holds
		\[
		(\prox_{r\tilde\phi}(u)-u, v-\prox_{r\tilde\phi}(u)) + \frac{1}{2} \|\prox_{r\tilde\phi}(u) - v\|^2 \ge r \tilde\phi(\prox_{r\tilde\phi}(u)) - r \tilde\phi(v).
		\]
	\end{lemma}
	\begin{proof}
		From the definition of the proximal operator we obtain
		\[
		\tilde\phi(v) + \frac1{2r}\|v-u\|^2 \ge \tilde\phi(\prox_{r\tilde\phi}(u)) + \frac1{2r} \| \prox_{r\tilde\phi}(u) - u\|^2 \qquad \forall v\in U.
		\]
		Reordering and using $\| v-u\|^2 = \| \prox_{r\tilde\phi}(u) -v\|^2 + 2(\prox_{r\tilde\phi}(u) - u, v- \prox_{r\tilde\phi}(u)) + \| \prox_{r\tilde\phi}(u) -u\|^2$ yields
		\begin{multline}
			\tilde\phi(v) - \tilde\phi(\prox_{r\tilde\phi}(u)) \ge \frac1{2r} \left( \| \prox_{r\tilde\phi}(u) - u\|^2 - \| v -u\|^2 \right) 
			\\= \frac1{2r}\left[ -\|\prox_{r\tilde\phi}(u)-v\|^2 - 2(\prox_{r\tilde\phi}(u) - u, v-\prox_{r\tilde\phi}(u))\right].
		\end{multline}
		Multiplying by $-r$ yields the claim.
	\end{proof}

	\begin{lemma}\label{lm:phiPsi}
		Let  $\tilde\phi: U\to \bar \R$ be some proper and lower semicontinuous function, let $u, d \in U$, $r>0$ and define 
		\begin{align*}
			\Phi(r) &\coloneq \| \prox_{r\tilde\phi}(u + rd) - u\|, \label{eq:phiDef} \\
			\Psi(r) &\coloneq \Phi(r)/r .
		\end{align*}
		\begin{enumerate}
			\item Then the mapping $r\mapsto \Phi(r) $ is nondecreasing for $r>0$.
			\item For all $r>0$, it holds 
			\begin{equation}\label{eq:proxIneqLm2}
				(d, \prox_{r\tilde\phi}(u+rd) - u) + \tilde\phi(u) - \tilde\phi(\prox_{r\tilde\phi}(u + rd)) \ge \frac12 \Psi(r) \Phi(r).
			\end{equation}
		\end{enumerate}	
	\end{lemma}
	\begin{proof}
		The proof of \cite[Theorem 10.9]{263firstOrderOpt} (c.f. \cite[Lemma 2]{240proxTrustRegion}) can be adapted to the nonconvex case using Lemma \ref{lm:lm1proxIneq} instead of the corresponding inequality for the convex case. Let $r_1>r_2$. By setting $r=r_1$, $u=u+r_1 d$ and $v = \prox_{r_2 \tilde\phi}(u+r_2 d)$ in Lemma \ref{lm:lm1proxIneq} we obtain 
		\begin{multline}
			(\prox_{r_1\tilde\phi}(u+r_1 d)-(u+r_1 d), \prox_{r_2 \tilde\phi}(u+r_2 d)-\prox_{r_1\tilde\phi}(u+r_1d)) \\ + \frac{1}{2} \|\prox_{r_1\tilde\phi}(u+r_1d) - \prox_{r_2 \tilde\phi}(u+r_2 d)\|^2 \\
			\ge r_1 \tilde\phi(\prox_{r_1\tilde\phi}(u+r_1d)) - r_1 \tilde\phi(\prox_{r_2 \tilde\phi}(u+r_2 d)).
		\end{multline}
		Let us define 
		\[
		p_{u,d}(r) \coloneq \prox_{r\tilde\phi}(u+rd)-u.
		\]
		Then, this can be written as
		\begin{multline*}
			(p_{u,d}(r_1)-r_1 d, p_{u,d}(r_2)-p_{u,d}(r_1) ) + \frac12 \| p_{u,d}(r_1)-p_{u,d}(r_2)\|^2 \\
			\ge r_1 \tilde\phi(\prox_{r_1\tilde\phi}(u+r_1d)) - r_1 \tilde\phi(\prox_{r_2 \tilde\phi}(u+r_2 d)).
		\end{multline*}
		Interchanging the roles of $r_1$ and $r_2$ yields
		\begin{multline*}
			(p_{u,d}(r_2)-r_2 d, p_{u,d}(r_1)-p_{u,d}(r_2) ) + \frac12 \| p_{u,d}(r_2)-p_{u,d}(r_1)\|^2 \\
			\ge r_2 \tilde\phi(\prox_{r_2\tilde\phi}(u+r_2d)) - r_2 \tilde\phi(\prox_{r_1 \tilde\phi}(u+r_1 d)).
		\end{multline*}
		We add these two equations after dividing by $r_1$ and $r_2$, respectively, to obtain 
		\[
		\left(\frac1{r_1}p_{u,d}(r_1)-\frac{1}{r_2} p_{u,d}(r_2), p_{u,d}(r_2)-p_{u,d}(r_1)\right) + \left(\frac{1}{2r_1}+\frac1{2r_2}\right) \|p_{u,d}(r_1) - p_{u,d}(r_2)\|^2 \ge 0.
		\]
		Expanding the two terms and reordering gives
		\[
		\frac{r_2-r_1}{2r_1 r_2}\|p_{u,d}(r_1)\|^2 + \frac{r_1-r_2}{2r_1 r_2}\|p_{u,d}(r_2)\|^2 \le 0,
		\]
		so $\|p_{u,d}(r_2)\| \le \|p_{u,d}(r_1)\|$ and the first part is  proven. \\
		As done in \cite[Lemma 2]{240proxTrustRegion}, we set $u = u+rd$ and $v=u$ in Lemma \ref{lm:lm1proxIneq} to obtain
		\begin{multline*}
			(\prox_{r\tilde\phi}(u+rd)-(u+rd), u-\prox_{r\tilde\phi}(u+rd)) + \frac{1}{2} \|\prox_{r\tilde\phi}(u+rd) - u\|^2 \\
			\ge r \tilde\phi(\prox_{r\tilde\phi}(u+rd)) - r \tilde\phi(u).
		\end{multline*}
		Reordering leads to \eqref{eq:proxIneqLm2}.
	\end{proof}
	Note that the additional quadratic term in Lemma \ref{lm:lm1proxIneq} leads to the additional constant $\frac12$ on the right hand side of inequality \eqref{eq:proxIneqLm2} compared to \cite[Lemma 2, part 3]{240proxTrustRegion}.
	\begin{remark}
		In the convex case, the function $\Psi$ defined in the previous Lemma is nonincreasing. This does not hold in the nonconvex case, as $r\mapsto \prox_{r\tilde\phi}(u+rd)$ is not continuous anymore as in the convex case by \cite[Lemma 3]{240proxTrustRegion}. This can be seen for example with $d=0$ and $\tilde\phi$ being the $L^p$-pseudonorm around some $r>0$ where $\prox_{r\phi}(u)$ jumps from zero to some nonzero value as $r$ decreases, compare Lemma \ref{lm:sparsityProxLp} later on. 
	\end{remark}

	\subsection{Optimality conditions}\label{sec:optConds}
	In this section we investigate optimality conditions for problem \eqref{eq:minprob} which will later serve as a basis for the convergence analysis of the trust-region algorithm.
	\subsubsection{In \texorpdfstring{$U=\Ltwo$}{U=L2}}
	We start by considering necessary optimality conditions for problem \eqref{eq:minprob} using the Fréchet and limiting subdifferential for the case $U=\Ltwo$. \\
	\begin{definition}
		The Fréchet subdifferential of a function $g: U \to \bar\R$ at $\bar u$ is defined by
		\[
		\hat\partial g(\bar u) \coloneqq \left\{ \eta\in U^* | \liminf_{\|h\|\searrow 0} \frac{g(\bar u + h)-g(\bar u) - (\eta,h)}{\|h\|} \ge 0 \right\}.
		\]
		The limiting subdifferential of $g$ is given by 
		\begin{multline*}
			\partial g(\bar u) \coloneqq \left\{ \eta \in U^* | \exists (u_k)_k \subset U, \exists (\eta_k)_k \subset U^* : \right. \\
			\left. u_k\to \bar u \text{ in } U, g(u_k) \to g(\bar u), \eta_k \rightharpoonup \eta  \text{ in } U^*, \eta_k\in \hat\partial g(u_k)\, \forall k\in \N \right\}.
		\end{multline*} 
	\end{definition} 
	In the unconstrained case $\Uad=U$ and under Assumption  \ref{ass:fPhi}, part \ref{ass:fPhifPart}, the sum rule applies to the limiting and Fréchet subdifferential by \cite[Theorem 3.36, Proposition 1.107]{genDiff60}. Thus, by Fermat's rule, a solution $\bar u$ of the minimization problem \eqref{eq:minprob} satisfies the necessary optimality conditions
	\begin{equation}\label{eq:subdiffInclusion}
		-\nabla f(\bar u) \in \hat\partial \phi(\bar u) \qquad \text{and} \qquad -\nabla f(\bar u) \in \partial \phi(\bar u).
	\end{equation}
	
	The inclusions \eqref{eq:subdiffInclusion} can be very weak. For example, at $\bar u=0$, both the Fréchet and limiting subdifferential of $\phi(u)=|u|^p$ with $p\in (0,1)$ are the whole space $\Ltwo$ as can be seen from the following lemma from \cite{38subdiffLp}.
	
	\begin{lemma}\cite[Theorem 4.6]{38subdiffLp} \label{lm:LpSubdiff}
		The Fréchet and limiting subdifferentials of $j: \Ltwo \to \R$, $j(u)\coloneq \io |u|^p \dd x$ for $p\in (0,1)$ are given by 
		\[
		\hat\partial j(u)  = \partial j(u) = \{ \eta \in L^2(\Omega): \eta = p|u|^{p-2}u  \text{ a.e. on } \{u\neq 0\}\}. 
		\]
	\end{lemma}
	
	Moreover, if $\phi$ is the sum of two functions that are not so-called sequentially normally epi-compact functions, then one can not apply the sum rule from \cite{genDiff60}. This is for example relevant if $\phi$ is the sum of the $L^p$-pseudo norm $\io |u|^p \dd x$ for $p\in (0,1)$ and the indicator function corresponding to control constraints. \\
	However, for the special case of box constraints, one can derive a sum rule as already suggested in \cite{38subdiffLp}.
	This fact explains definition \eqref{def:Uad} of the admissible set $\Uad$ in Assumption \ref{ass:fPhi}. \\
	By definition of the convex subdifferential, it holds 
	\begin{align*}
		\partial \Iuad (\bar u)  &= \{v\in \Ltwo: 0 \ge (v, w-\bar u)_{\Ltwo} \quad \forall w \in \Uad\} \\
		&= \{ v \in \Ltwo: v \ge 0 \,\text{ a.e. on } \{\bar u = b\}, v \le 0 \, \text{ a.e. on } \{\bar u = a\}, v=0 \text{ elsewhere}\}.
	\end{align*}
	The Fréchet and limiting subdifferentials coincide with the convex subdifferential for convex functions, see e.g. \cite[page 83, 95]{genDiff60}.
	\begin{lemma}\label{lm:FrechetSubdiffSumrule}
		It holds $\hat \partial (j(\bar u) +I_{\Uad}(\bar u) ) = \hat \partial j(\bar u) + \hat \partial \Iuad(\bar u) $ for all $\bar u\in \Uad$. I.e.,
		\[
		\hat\partial (j(\bar u) +I_{\Uad}(\bar u) ) = \left\{ \eta \in \Ltwo: \eta \begin{cases} \ge p|\bar u|^{p-2}\bar u &  \text{a.e. on } \{\bar u = b\}, \\
			\le p|\bar u|^{p-2}\bar u & \text{a.e. on } \{\bar u = a\}, \\    
			= p|\bar u|^{p-2}\bar u & \text{a.e. elsewhere on  } \{\bar u \neq 0\}
		\end{cases}   \right\}.
		\]        
	\end{lemma}
	
	\begin{proof}
		The proof is an extension of some proofs in \cite{38subdiffLp}. \\
		We start with the inclusion "$\subseteq$" following \cite[Lemma 4.1]{38subdiffLp}. Let $\eta \in \hat \partial (j + \Iuad)(\bar u)$, let $\epsilon>0$ and define the set $A_{\epsilon}^- \coloneq \{|\bar u|>\epsilon\} \cap \{\bar u<b-\epsilon\}$. Let $B^-\subset A_{\epsilon}^-$ be a measurable set of positive measure and define $(h_k)_k \subset L^2(\Omega)$ as $h_k \coloneq \frac1{k}\chi_{B^-}$ for all $k\in \N$. Thus, it holds $\|h_k\| \searrow 0$ and for $k$ large enough one has $\bar u + h_k \in \Uad$. 
		By definition of the Fréchet subdifferential we obtain
		\begin{multline*}
			0 \le |B^-|^{1/2} \liminf_{k\to\infty} \frac{(j+\Iuad)(\bar u + h_k) - (j+\Iuad)(\bar u)-\io \eta(x) h_k(x) \dd x}{\|h_k\|} \\
			=  |B^-|^{1/2} \liminf_{k\to\infty} \frac{j(\bar u + h_k) - j(\bar u)-\io \eta(x) h_k(x) \dd x}{\|h_k\|} \\
			= \liminf_{k \to \infty} \int_{B^-} \left( k \left( \left| \bar{u}(x) + \frac{1}{k} \right|^p - |\bar{u}(x)|^p \right) - \eta(x) \right)  \dd x = \int_{B^-} p |\bar{u}(x)|^{p-2} \bar{u}(x) - \eta(x) \, \dd x
		\end{multline*}
		via dominated convergence. 
		Using the set $A_{\epsilon}^+\coloneq \{|\bar u|>\epsilon\} \cap \{\bar u> a+\epsilon\}$, one can repeat the steps above with $\tilde h_k = - \frac1{k}\chi_{B^+(x)}$ for subsets $B^+ \subset A_{\epsilon}^+$ to obtain a reverse inequality. As the sets $B^+$ and $B^-$ were arbitrary, this shows $\eta \le p |\bar{u}|^{p-2} \bar{u}$ a.e. on $A_{\epsilon}^-$ and  $\eta \ge p |\bar{u}|^{p-2} \bar{u}$ a.e. on $A_{\epsilon}^+$. Using $ \{|\bar u|\neq 0 \} \cap \{\bar u<b\}= \cup_{\epsilon>0} A_{\epsilon}^-$ and $ \{|\bar u|\neq 0 \} \cap \{\bar u>a\}= \cup_{\epsilon>0} A_{\epsilon}^+$ and combining both inequalities on $\{|\bar u|\neq 0 \} \cap \{a<\bar u<b\}$ yields the claim. \\
		For the reverse inclusion "$\supseteq$", we proceed similarly as in \cite[Theorem 4.4]{38subdiffLp}. Let $\eta \in L^2(\Omega)$ with $\eta = p|\bar u|^{p-2}\bar u \chi_{ \{\bar u \neq 0\}} + \eta_1 \chi_{\{\bar u = b\}} - \eta_2 \chi_{\{\bar u = a\}}$ with functions $0\le\eta_1,\eta_2$. Let $(h_k)_k\subset L^2(\Omega)$ with $\|h_k\| \searrow 0$. We want to show
		\[
		\liminf_{k\to\infty} \frac{(j+\Iuad)(\bar u + h_k) - (j+\Iuad)(\bar u) - \io \eta h_k \dd x}{\|h_k\|} \ge 0.
		\] 
		If $\bar u + h_k \notin \Uad$ for all $k$, the statement clearly holds as $\bar u \in \Uad$, so we can assume w.l.o.g. that $\bar u + h_k \in \Iuad$ for $k$ large enough. Thus, we have to show
		\begin{multline*}
			0\le \liminf_{k\to\infty} \frac{j(\bar u + h_k) - j(\bar u) - \io \eta h_k \dd x}{\|h_k\|} \\
			= \liminf_{k\to\infty} \frac{\int_{ \{\bar u = 0\}}(|h_k(x)|^p-\eta(x) h_k(x)) \dd x}{\|h_k\|} + \frac{\int_{\{\bar u \neq 0\}}(|\bar u(x) + h_k(x))|^p - |\bar u(x)|^p - \eta(x) h_k(x) \dd x}{\|h_k\|}    \\
			\eqqcolon \liminf_{k\to\infty} D_k^1 + D_k^2.
		\end{multline*}
		As $D_k^1$ coincides with the unconstrained case from \cite[Theorem 4.4]{38subdiffLp}, it holds $D_k^1\ge 0$ and it  remains to show $D_k^2\ge 0$.
		Arguing as in \cite[Lemma 4.2]{38subdiffLp}, we obtain
		\[
		\liminf_{k\to \infty } \tilde D_k^2 \coloneqq \liminf_{k \to \infty} \frac{\int_{\{ \bar{u} \neq 0 \}} \left( |\bar{u}(x) + h_k(x)|^p - |\bar{u}(x)|^p - p |\bar{u}(x)|^{p-2} \bar{u}(x) h_k(x) \right) \dd x}{\| h_k \|} \geq 0.	
		\]
		Since $\bar u + h_k \in \Uad$, it holds $h_k \ge 0$ a.e. on $\{\bar u = a\}$ and $h_k \le 0$ a.e. on $\{\bar u = b\}$. Thus, the remaining terms in $D_k^2 - \tilde D_k^2$ are also non-negative, which finishes the proof.
	\end{proof}
	
	\begin{lemma}\label{lm:LpSubdiffIndicator}
		For all $u\in \Uad$, it holds 
		\begin{multline*}
			\partial ((j(u) +I_{\Uad}(u) ) ) = \hat \partial (j(u) +I_{\Uad}(u) ) 
			= \hat\partial j(u) + \hat\partial I_{\Uad}(u) = \partial j(u) +\partial I_{\Uad}(u)  \\
			= \left\{ \eta \in \Ltwo: \eta \begin{cases} \ge p|\bar u|^{p-2}\bar u &  \text{a.e. on } \{\bar u = b\}, \\
				\le p|\bar u|^{p-2}\bar u & \text{a.e. on } \{\bar u = a\}, \\    
				= p|\bar u|^{p-2}\bar u & \text{a.e. elsewhere on  } \{\bar u \neq 0\}.
			\end{cases}   \right\}
		\end{multline*}
		
	\end{lemma}
	\begin{proof}
		Note that by \cite[Theorem 4.6]{38subdiffLp} and convexity of $\Iuad$ it holds $\partial j(u) +\partial I_{\Uad}(u) = \hat\partial j(u) + \hat\partial I_{\Uad}(u) $. By Lemma \ref{lm:FrechetSubdiffSumrule} it remains to prove the first equality. 
		We follow \cite[Theorem 4.6]{38subdiffLp} and note that the inclusion "$\supseteq$" follows directly from the definition of the limiting and Fréchet subdifferentials. 
		For "$\subseteq$", we proceed similarly as in \cite[Theorem 3.7]{38subdiffLp}. Let $\eta \in \partial(j+\Iuad)(\bar u)$ and let $(u_k)_k \subset L^2(\Omega)$, $(\eta_k)_k \subset L^2(\Omega)$ with $u_k \to \bar u$ in $L^2(\Omega)$, $(j+\Iuad)(u_k) \to (j+\Iuad)(\bar u)$, $\eta_k \rightharpoonup \eta$ in $L^2(\Omega)$ and $\eta_k \in \hat \partial (j+\Iuad)(u_k)$. After possibly extracting a subsequence, we can assume that $u_k(x)\to \bar u(x)$ converges pointwise a.e. on $\Omega$. Furthermore, by $(j+\Iuad)(u_k) \to (j+\Iuad)(\bar u)$ we can also assume $u_k \in \Uad$ for all $k$.
		Let $x\in \{ \bar u \neq 0\} \cap \{ \bar u \in (a,b)\}$. Then $x\in \{ u_k \neq 0\} \cap \{ u_k \in (a,b)\}$ for $k$ large enough.
		Thus, it holds $\eta_k = p|u_k|^{p-2}u_k$ which implies $\eta_k(x) \to p |\bar u|^{p-2}\bar u$ a.e. on $\{ \bar u \neq 0\} \cap \{ \bar u \in (a,b)\}$. 
		Since weak and pointwise limits coincide, we obtain $\eta = p |\bar u|^{p-2}\bar u$ a.e. on $\{ \bar u \neq 0\} \cap \{ \bar u \in (a,b)\}$. \\
		Let now $x\in \{ \bar u = a\}$. Let us split $(u_k(x))$ into two subsequences with $(u_{k_n^1}(x))_{k_n^1} < a$ and $(u_{k_n^2}(x))_{k_n^2} = a$. For the corresponding subsequence $\eta_{k_n^1}$, it holds $\eta_{k_n^1}(x) = p|u_{k_n^1}(x)|^{p-2}u_{k_n^1}(x) \to p|\bar u(x)|^{p-2}\bar u(x) = p|a(x)|^{p-2}a(x)$, analogously to the case above. For the second subsequence we have $\eta_{k_n^2}(x) \ge p |a(x)|^{p-2}a(x)$. Combining these two estimates and passing to the limit inferior yields $\liminf_{k\to \infty} \eta_k(x) \ge p |a(x)|^{p-2}a(x)$, so $\eta \ge p|\bar u|^{p-2}\bar u$ a.e. on $x\in \{ \bar u = a\}$ using Fatou's lemma. 
		One can proceed similarly to obtain the corresponding inequality on $x\in \{ \bar u = b\}$.
	\end{proof}
	The previous lemmas will be useful later on when proving criticality results of the trust-region algorithm.\\
	
	A stronger optimality condition is the so-called $L$-stationarity (see e.g. \cite{15proxNonsmooth, sparsityConstrainedOpt308, introNonlinOpt309})  which is satisfied if  
	\begin{equation}\label{eq:Lstationary}
		\bar u \in \prox_{1/L \phi}\left(\bar u - \frac1{L} \nabla f(\bar u)\right).
	\end{equation}
	For convex and lower semicontinuous functions $\phi$, $L$-stationarity and \eqref{eq:subdiffInclusion} are equivalent, but in the nonconvex case the inclusion \eqref{eq:Lstationary} is stronger, see e.g. \cite[Introduction]{15proxNonsmooth}. \\
	In \cite[Theorem 2.1]{15proxNonsmooth} it was shown for $a=-b$ in $\Uad$ being constant that local solutions $\bar u \in \Linf$ of \eqref{eq:minprob} with $f(u)-f(\bar u) = \nabla f(\bar u)\cdot (u-\bar u) +o(\|u-\bar u\|)$ satisfy the Pontryagin maximum principle. The Pontryagin maximum principle also implies $L$-stationarity.
	
	\subsubsection{In \texorpdfstring{$U=\Hone$}{U=H01}}
	Next, we consider problem \eqref{eq:minprob} in $\Hone$. We first note that since $0<p<1$ and $\Omega$ is bounded, we obtain compactness of the embedding from $\Hone$ into $L^p(\Omega)$ by the Rellich-Kondrachov theorem or via Morrey's inequality for all dimensions $d$. Due to this compact embedding, the $L^p$-functional is now weakly lower semicontinuous in this case. This guarantees existence of solutions by standard methods, see for example \cite{lp_cont, sparse}.\\  
	A necessary optimality condition for this problem was derived in \cite[section 5]{sparse}.
	\begin{theorem}\cite[Theorem 5.7]{sparse}\label{tm:ocH1}
		Let $\bar u$ be a local solution of \eqref{eq:minprob} with  $\alpha>0$. Then there exists $\bar \lambda \in \Hone^*$ such that
		\[
		\alpha (\bar u,v)_{\Hone} + \beta \braket{\bar\lambda,v}_{\Hone} = -f'(\bar u) v \qquad \forall v\in \Hone.
		\]
		with 
		\[
		\braket{\bar \lambda,\bar u} = p \io |\bar u|^p \dd x.
		\]
	\end{theorem}	

	\subsection{Smoothing of \texorpdfstring{$L^p$}{Lp}-pseudonorm}
	In the trust-region algorithm, we will use a smooth approximation of the $L^p$-pseudo-norm that was introduced in \cite{lp_cont}.
	There, a smooth approximation of $t\mapsto t^{p/2}$ is defined via the function
	\[
	\psi_{\epsilon}(t) \coloneqq \begin{cases}
		\frac{p}{2}\frac{t}{\epsilon^{2-p}}+(1-\frac{p}{2})\epsilon^p & \textrm{if } t \in [0,\epsilon^2), \\ t^{p/2} & \textrm{if } t \geq \epsilon^2.
	\end{cases}
	\]
	Then, 
	\begin{equation}\label{eq:defTildePhiEps}
		j_{\epsilon}(u) \coloneqq \io \psi_{\epsilon}(u^2) \dd x \approx \io |u|^p \dd x
	\end{equation}
	is an approximation of the $L^p$-pseudonorm.
	
	The following results from \cite{sparse} will be useful later on.
	\begin{lemma}
		\label{lm:psieps} Let $p\in(0,2)$. Then $\psi_{\epsilon}: [0, \infty) \to [0,\infty)$
		has the following properties:
		\begin{enumerate}[label=(\roman*), series=psi]
			\item The function $\psi_{\epsilon}$ is concave. 
			\item $\epsilon \mapsto \psi_{\epsilon}(t)$ is monotonically increasing for all $t\geq 0$.
			\item Let $\epsilon>0$. Then $\psi$ is continuously differentiable with
			\[ \psi_{\epsilon}'(t)=\frac{p}{2}\min(\epsilon^{p-2},t^{\frac{p-2}{2}}).
			\] 
		\end{enumerate}
		Let now $p\in(0,1)$.
		\begin{enumerate}[label=(\roman*), resume=psi]
			\item For sequences $(u_k)$, $(\epsilon_k)$ such that $u_k \to u$ in $L^1(\Omega)$ and $\epsilon_k \to \epsilon \geq 0$ it holds $j_{\epsilon_k}(u_k) \to j_{\epsilon}(u)$.
			\item Let $\epsilon>0$. Then $u \mapsto G_{\epsilon}(u)$ is Fréchet differentiable from $L^2(\Omega)$ to $\R$ with \begin{align*}
				j_{\epsilon}'(u)h = \int_{\Omega} 2 u(x) \psi_{\epsilon}'(u(x)^2) h(x) \dd x.
			\end{align*}
			\item For $u_k, u\in \Ltwo$ and $\epsilon_k>0$ define 
			\[
			j_k(u) \coloneqq \io \psi_{\epsilon_k}(u_k^2) + \psi_{\epsilon_k}'(u_k^2)(u^2-u_k^2) \dd x. 
			\]
			Then $j_k(u)\ge j_{\epsilon_k}(u)$
		\end{enumerate}
	\end{lemma}
	\begin{proof}
		The first five statements were proven in \cite[Lemma 4.1, 4.2, 4.3]{sparse}. By concavity and differentiability of $\psi_{\epsilon}$ it holds 
		\[
		\psi_{\epsilon_k}(t^2) + \psi_{\epsilon_k}'(t^2)(s^2-t^2) \ge \psi_{\epsilon}(s^2)
		\]
		for $s,t\in\R$. Using this inequality pointwise in $j_k$ yields the last claim.
	\end{proof}
	
	Note that the fact that $j_k$ majorizes $j_{\epsilon}$ will turn out to be a crucial property in the analysis of the trust-region algorithm. 	\\
	In the following, we will consider the functions $\phi_k$ and $\phi_{\epsilon}$ defined by 
	\begin{equation}\label{eq:defPhik}
	\phi_k(u) = \beta j_k(u) + \frac{\alpha}{2}\|u\|^2 + \Iuad \qquad \text{and} \qquad \phi_{\epsilon}(u) = \beta j_{\epsilon}(u) + \frac{\alpha}{2}\|u\|^2 + \Iuad
	\end{equation}
	for some $\alpha\ge0, \beta>0$ and $\Uad$ defined as in \eqref{def:Uad}.
	
	In the $\Ltwo$-setting in \cite{15proxNonsmooth} it was shown that regularizing with $L^p$ for $p\in (0,1)$ induces sparsity and that proximal points $u$ satisfy either $u=0$ or $u\ge u_0$ for some $ u_0>0$ (\cite[Theorem 3.7]{15proxNonsmooth}). The following lemma aims to obtain similar results for the smooth approximation $\phi_{\epsilon}$.
	
	\begin{lemma}\label{lm:sparsityProxLp}
		Let $U=\Ltwo =\Uad$, $\epsilon\ge 0$ and $v \in \prox_{r\phi_{\epsilon}}(\bar u)$. Then it holds $|v|\le \epsilon$ or $|v|\ge u_0(r) \coloneq \sqrt[2-p]{\frac{r\beta p(1-p)}{1+\alpha r}}$ a.e. on $\Omega$.
	\end{lemma}
	\begin{proof}
		Assume $\bar u \ge0$, so also $v\ge0$.  The case $\bar u\le 0$ can be treated analogously by symmetry.
		On $(\epsilon, \infty)$, the objective $v\mapsto \Phi(v) \coloneq \beta \psi_{\epsilon}(v^2) + \frac{\alpha}{2}v^2 +\frac1{2r} (v-\bar u)^2$ used to compute $\prox_{r\phi_{\epsilon}}(\bar u)$ pointwise is twice continuously differentiable with second derivative $\Phi''(v) = \frac{1}{r}+\alpha+p(p-1)v^{p-2}$. This function in monotonically increasing with a unique root at $u_0(r)$ as defined above, so $\Phi$ has a unique positive inflection point on $(\epsilon, \infty)$ if $u_0(r)>\epsilon$. Assume now $|v|>\epsilon$. As it has to hold $\Phi''(v)\ge 0$ at a local minimum and we have $\Phi''(v)<0$ for $|v|<u_0(r)$, the claim follows. 
	\end{proof}

	\begin{remark}\label{rm:proxNotSingleVal}
		Note that the proximity operator for $\phi_{\epsilon}$ with $\epsilon\ge 0$ is not single valued valued, as both the local minima with $|v|\le \epsilon$ or $|v|\ge u_0(r)$ can be global minima of the proximal minimization problem.
	\end{remark}

    \begin{remark} 
        For $U=\Hone$, a result like Lemma \ref{lm:sparsityProxLp} is not available, at least not for $d\le 3$. This is because minimizer of the proximal optimization problem are Hölder-continuous on $\Omega$ by \cite[Theorem 3.1]{regOfMinima348}, where we have used that they are also in $\Linf$ by \cite[Theorem 5.12]{sparse}. 
    \end{remark}
	
	\section{Trust-region algorithm with smoothing}\label{sec:algo}
	We aim to solve problem \eqref{eq:minprob} with a trust-region algorithm by generalizing the algorithm from \cite{240proxTrustRegion} for the convex case.
	That algorithm is an inexact proximal trust-region algorithm that solves the composite problem 
	\[
	\min_{u\in U}  f(u) + \phi(u)
	\]
	for smooth $f$ and convex $\phi$. It is based on the computation of a generalized Cauchy point (GCP) along a proximal path. This GCP is easy to compute and satisfies a sufficient decrease condition in the convex case. Therefore, we also aim to use this generalized Cauchy path in the nonconvex case. However, the proof of sufficient decrease can not be directly generalized to our $L^p$-setting here. To overcome this difficulty and to reuse some of the results from the convex case, we work with a convex upper bound of the objective in the trust-region subproblems using the function $\phi_k$ from \eqref{eq:defPhik}.\\
	We start by some definitions. Let $h_k$ be defined as
	\[
	h_k \coloneq \frac1{r_0}\|\prox_{r_0\phi_k}(u_k-r_0\nabla f_k(u_k)) - u_k\|
	\]
	for some $r_0>0$, and let
	\[
	h_k(u) \coloneqq \frac1{r_0}\|\prox_{r_0 \phi_k}(u-r_0\nabla f(u)) -u\|.
	\]
	Note that the convex function $\phi_k$ is used in these definitions, so the proximity operator is single valued. \\
	In each trust-region step, we compute a solution of the trust-region subproblem. We will consider two subproblems defined by the models 
	\[
	m_k(u) \coloneq f_k(u) + \phi_k(u) \qquad \text{and} \qquad m_k^{\epsilon}(u) \coloneq f_k(u) + \phi_{\epsilon_k}(u).
	\]
	The trust-region subproblems then read
	\[
	\min_{u\in U} m_k^{(\epsilon)} \qquad \text{subject to} \qquad \|u-u_k\| \le \Delta_k,
	\]
	where $\Delta_k$ denotes the trust-region radius.
	Here, $f_k$ is a local model of $f$ around $u_k$ that satisfies the following inexactness conditions from \cite{240proxTrustRegion}.
	\begin{assumption}\label{ass:inexactGradient}
		The model $f_k$ is $L_k$ smooth (on the domain of $\phi$) and the gradient $g_k \coloneqq \nabla f_k(u_k)$ satisfies
		\[
		\|g_k - \nabla f(u_k)\| \le \kappa_{\text{grad}}\min(h_k,\Delta_k) \qquad \forall k\in \N,
		\]
		where $\kappa_{\text{grad}}\ge 0$ is independent of $k$.
	\end{assumption}
	We want to approximate solutions of those subproblems with $m_k$ or $m_k^{\epsilon}$ by computing points that satisfy the fraction of Cauchy decrease (FCD) condition
	\begin{equation}\label{eq:fcdCond}
		m_k^{(\epsilon)}(u_k)-m_k^{(\epsilon)}(u\kp ) \ge \kappa_{\text{fcd}}h_k \min\left(\frac{h_k}{1+\omega_k}, \Delta_k\right)
	\end{equation}
	and that lie inside the trust-region, i.e.
	\begin{equation}
		\|u\kp-u_k\| \le \kappa_{\text{rad}}\Delta_k,
	\end{equation}
	where the constants $\kappa_{\text{fcd}}, \kappa_{\text{rad}}>0$ are both independent of $k$. 
	\begin{remark}\label{rm:fcdCondNonconv}
		Assume the FCD condition \eqref{eq:fcdCond} holds for the convex model $m_k$. As $\phi_k$ majorizes $\phi_{\epsilon_k}$, it holds $m_k(u\kp ) \ge m_k^{\epsilon} (u\kp ) $. Substituting this into the FCD condition \eqref{eq:fcdCond} with the convex model $m_k$ and using $m_k(u_k)=m_k^{\epsilon}(u_k)$ then also implies the FCD
		condition \eqref{eq:fcdCond} with the nonconvex model $m_k^{\epsilon}$. 
		\end{remark}
	In trust-region algorithms, the trust-region radius is updated after each iteration according to the ratio of actual and predicted reductions. Thus, for the two models, let
	\[
	\ared_k = f(u_k) + \phi_k(u_k) - f(u\kp) - \phi_k(u\kp), \quad \ared_k^{\epsilon} = f(u_k) + \phi_{\epsilon_k}(u_k) - f(u\kp) - \phi_{\epsilon_k}(u\kp), 
	\]
	\[
	\pred_k \coloneq m_k(u_k)-  m_k(u\kp), 
	\quad \pred_k^{\epsilon} = m_k^{\epsilon}(u_k) - m_k^{\epsilon}(u\kp)
	\]
	and
	\begin{equation}
		\tilde\rho_k \coloneq \frac{\ared_k}{\pred_k} \qquad \text{and} 
		\qquad \tilde\rho_k^{\epsilon} \coloneq \frac{\ared_k^{\epsilon}}{\pred_k^{\epsilon}}.
	\end{equation}
	As mentioned in \cite{240proxTrustRegion}, it is not always possible to compute the objective $F$ exactly, or here also its smoothed version involving $j_{\epsilon}$. Thus, we introduce the computed reduction $\cred_k^{\epsilon}$, i.e. the numerically computed approximation of $\ared_k^{(\epsilon)}$. Then we can replace $\ared_k^{(\epsilon)}$ by the computed reduction $\cred_k^{(\epsilon)}$, and we work with the ratios of computed and predicted reductions
	\begin{equation}\label{eq:defRhos}
		\rho_k \coloneq \frac{\cred_k}{\pred_k} \qquad \text{and} 
		\qquad \rho_k^{\epsilon} \coloneq \frac{\cred_k^{\epsilon}}{\pred_k^{\epsilon}}
	\end{equation}
	in the algorithm.
	In order to ensure that $\cred_k^{(\epsilon)}$ approximates $\ared_k^{(\epsilon)}$ well enough, we make the following assumption as in \cite{240proxTrustRegion, 251inexactObjTR}.
	\begin{assumption}\label{ass:inexObj}
		There exists $\kappa_{\text{obj}} \ge 0$ independent of $k$ such that 
		\[
		|\ared_k^{(\epsilon)} - \cred_k^{(\epsilon)}| \le \kappa_{\text{obj}}(\eta \min(\pred_k^{(\epsilon)}, \theta_k))^{\zeta} \qquad \forall k,
		\]
		where $\zeta>1$, $0<\eta<\min(\eta_1, 1-\eta_2)$ and $\lim_{k\to\infty} \theta_k = 0$.
	\end{assumption}
	Convergence $\theta_k\to 0$ for $k\to\infty$ implies that $\theta_k\le \kappa_{\text{obj}}^{-1/(\zeta-1)}$ for $k$ large enough. Substituting this inequality and using $\min(\pred_k^{(\epsilon)}, \theta_k)) \le \theta_k$ yields
	\begin{equation}\label{eq:ineqAredCred}
		|\ared_k^{(\epsilon)} - \cred_k^{(\epsilon)}|\le \kappa_{\text{obj}}(\eta\min(\pred_k^{(\epsilon)}, \theta_k))\theta_k^{\zeta-1} \le \eta \min(\pred_k^{(\epsilon)}, \theta_k).
	\end{equation}
	We repeat lemma \cite[Lemma A.1]{251inexactObjTR} (= \cite[Lemma 6]{240proxTrustRegion}) to relate $\rho_k^{(\epsilon)}$ and $\tilde \rho_k^{(\epsilon)}$.
	\begin{lemma}\cite[Lemma A.1]{251inexactObjTR}\label{lm:inexactFEval}
		Let Assumption \ref{ass:inexObj} hold with some $0<\eta<\min(\eta_1, 1-\eta_2)$. Then there exists $K_{\eta} \in \N$ such that 
		\[
		\tilde\rho_k^{(\epsilon)} = \frac{\ared_k^{(\epsilon)}}{\pred_k^{(\epsilon)}} \in [\rho_k^{(\epsilon)}-\eta, \rho_k^{(\epsilon)} + \eta] \qquad \forall k \ge K_{\eta}.
		\]
	\end{lemma}
	
	\begin{algorithm}
		\caption{Smoothed proximal trust-region algorithm for $L^p$-regularized problems}\label{alg:smoothingTrForLp}
		\begin{algorithmic}[1]
			\State Choose a monotonically decreasing sequence $(\epsilon_k)$ with $\epsilon_k \searrow 0$, initial value $u_0$, parameters $\nu>1$, $0<\eta_1<\eta_2<1$, $\Delta_0>0$, $0<\gamma_1 \le \gamma_2 < 1 \le \gamma_3$. Set $k=0$.
			\While{ a suitable stopping criterion is not satisfied }
			\State \textbf{Model selection}: Choose models $m_k=f_k + \phi_k$ and $m_k^{\epsilon}=f_k + \phi_{\epsilon_k}$. 
			\State \textbf{Step computation:} Compute an approximate solution $u_{k+1}$ of one of the trust-region subproblems $m_k$ or $m_k^{\epsilon}$
			that satisfies \eqref{eq:fcdCond}  
			\State \textbf{Compute reductions } $\pred_k^{(\epsilon)}$ and $\cred_k^{(\epsilon)}$, using the reductions with $\epsilon$ in case of having chosen $m_k^{\epsilon}$ in the previous step.
			\State \textbf{Trust-region update: } Compute $\rho_k^{(\epsilon)} = \cred_k^{(\epsilon)}/\pred_k^{(\epsilon)}$. 
			\If{$\rho_k^{(\epsilon)} < \eta_1$}
			\State $u\kp = u_k$
			\State $\Delta\kp \in [\gamma_1 \Delta_k, \gamma_2 \Delta_k]$
			\ElsIf{$\rho_k^{(\epsilon)} \in [\eta_1, \eta_2) $}
			\State $\Delta\kp \in [\gamma_2\Delta_k, \Delta_k]$
			\Else
			\State $\Delta\kp \in [\Delta_k, \gamma_3 \Delta_k]$
			\EndIf
			\State Set $k=k+1$.
			\EndWhile
		\end{algorithmic}
	\end{algorithm}
	
	Algorithm \ref{alg:smoothingTrForLp} describes the smoothed majorized trust-region algorithm for problem \eqref{eq:minprob}. In each iteration, we compute a trial iterate satisfying the FCD condition \eqref{eq:fcdCond} using either the convex model $m_k$ or the nonconvex model $m_k^{\epsilon}$. Note that by definition of $h_k$, the convex upper bound $\phi_k$ is always involved. The remainder of the algorithm coincides with \cite[Algorithm 1]{240proxTrustRegion}. 
	As a stopping criterion we choose 
	\begin{equation}\label{eq:stopCrit}
		h_k<\tau_0 h_0 \qquad \text{ for some } \qquad \tau_0>0
	\end{equation}
	as in \cite{240proxTrustRegion}.

	Note that since $\phi_k$ is quadratic and convex, one can easily verify that for $U=\Ltwo$ it holds
	\begin{multline}\label{eq:formulaProxPhik}
		\prox_{r\phi_k}(u) = \arg \min_{a\le v\le b} \frac{1}{2r}\|v-u\|^2 + \frac{\alpha}{2}\|v\|^2  + \beta \io \psi_{\epsilon_k}'(u_k^2)v^2 \dd x \\
		= \proj_{[a,b]}\left(\frac{u}{1+\alpha r + 2r\beta \psi_{\epsilon}'(u_k^2)}\right).
	\end{multline}
	For $U=\Hone$, the minimization problem to compute $v=\prox_{r\phi_k}(u)$ is characterized by the first order optimality condition 
	\begin{equation}\label{eq:H1prox}
		\frac{1}{r} (v-u, w)_{\Hone} + 2\beta\io \psi_{\epsilon_k}'(u_k^2)vw \dd x + \alpha(v,w)_{\Hone} = 0 \qquad \forall w\in \Hone.
	\end{equation}
	Thus, the structure of the upper bound $\phi_k$ of the objective $\phi_{\epsilon_k}$ that is constructed in each step allows for an efficient proximal point computation when computing $h_k$. \\
	Later on in section \ref{sec:GCP} we will also use this property to compute points that satisfy the FCD condition \ref{eq:fcdCond}.

    \begin{remark}
        To keep the results as general as possible, we allow for either the convex or the nonconvex model to be used in each step of the algorithm to compute iterates satisfying the FCD condition \eqref{eq:fcdCond} and the trust-region radius $\rho_k^{(\epsilon)}$.  In practice, however, the nonconvex model often yields better results as it is closer to the original function to minimize, see section \ref{sec:numRes}.
    \end{remark}
	
	\subsection{Convergence analysis}
	Next, we investigate convergence properties of Algorithm \ref{alg:smoothingTrForLp} and generalize some convergence results from \cite{240proxTrustRegion} to our nonconvex $L^p$-setting. To do so, we assume that the curvature is bounded in a way such that
	\[
	\sum_{k=1}^{\infty} \frac{1}{b_k} = \infty,
	\]
	where $b_k := 1 + \max\{\omega_i \mid i = 1, \ldots, k\}$ with $\omega_k := \sup\{|\omega(f_k, u_k, s)| : 0 < \|s\|_U \leq \kappa_{\text{rad}} \Delta_k\}$ and the curvature $\omega$ from \eqref{eq:defCurvature}.
	
	Furthermore, we define sets of indices of successful iterations as
	\[
	\mathcal{S} := \{k \in \mathbb{N} \mid \rho_k^{(\epsilon)} \geq \eta_1\}, \quad 
	\mathcal{S}(n) := \{k \in \mathcal{S} \mid k < n\}.
	\]
	
	\begin{theorem}\label{tm:convHk}
		Let $(u_k)$ be a sequence generated by Algorithm \ref{alg:smoothingTrForLp} and let Assumptions \ref{ass:fPhi}, \ref{ass:inexactGradient} and \ref{ass:inexObj} be satisfied. Then it holds 
		\[
		\liminf_{k\to\infty} h_k = 0 \qquad \text{and} \qquad \liminf_{k\to\infty} h(u_k) = 0. 
		\]
	\end{theorem}
	\begin{proof} 
		Let $\rho_k>\eta_1$, let the FCD condition be satisfied with the convex model $m_k$ and let $k\ge K_{\eta}$ with $K_{\eta}$ from Lemma \ref{lm:inexactFEval}.  Then, 
		\begin{equation}\label{eq:auxHkConv}
		(\eta_1-\eta)\kappa_{\text{fcd}} h_k \min\left(\frac{h_k}{1+\omega_k}, \Delta_k\right) \le (\eta_1-\eta) \pred_k
		\le f(u_k) + \phi_k(u_k) - (f( u\kp) +  \phi_k( u\kp)), 
		\end{equation}
		where the first inequality follows from the FCD condition \eqref{eq:fcdCond} and the second one by Lemma \ref{lm:inexactFEval}.         
		By concavity of $\psi_{\epsilon}$ it holds $j_k(u\kp) = \io \psi_{\epsilon_k}(u_k^2) + \psi_{\epsilon_k}'(u_k^2)(u\kp^2-u_k^2) \ge \io \psi_{\epsilon_k}(u\kp^2) = j_{\epsilon_k}(u\kp) $. Monotonicity of $\epsilon \to \psi_{\epsilon}(u)$ from Lemma \ref{lm:psieps} yields  $j_{\epsilon_k}(u\kp) \ge j_{\epsilon\kp}(u\kp)=j\kp(u\kp)$, where the last equality follows by definition of $j\kp$ and $j_{\epsilon\kp}$. Combining these estimates yields $j_k(u\kp) \ge j\kp(u\kp)$, which implies $\phi_k(u\kp) \ge \phi\kp(u\kp)$.
		Thus, from \eqref{eq:auxHkConv} we obtain
		\[
		(\eta_1-\eta)\kappa_{\text{fcd}} h_k \min\left(\frac{h_k}{1+\omega_k}, \Delta_k\right) \le (\eta_1-\eta) \pred_k
		\le f(u_k) + \phi_k(u_k) - (f(u\kp) +  \phi\kp(u\kp)).
		\]
		If the model $m_k^{\epsilon}$ is used in the FCD condition and $\rho_k^{\epsilon}>\eta_1$, it also holds 
		\begin{multline*}
			(\eta_1-\eta)\kappa_{\text{fcd}} h_k \min\left(\frac{h_k}{1+\omega_k}, \Delta_k\right) 
			\le f(u_k) + \phi_{\epsilon_k}(u_k) - (f( u\kp) +  \phi_{\epsilon_k}( u\kp)) \\
			\le f(u_k) + \phi_k(u_k) - (f( u\kp) +  \phi\kp( u\kp)),
		\end{multline*}
		using the monotonicity argument for $\epsilon$ from above to obtain $\phi_{\epsilon_k}(u\kp) \ge \phi_{\epsilon\kp}(u\kp)=\phi\kp(u\kp)$.
		Thus, we obtain the same inequality in both cases. This can be used in the first chain of inequalities in the proof of \cite[Theorem 3]{240proxTrustRegion} to estimate the sum 
		\begin{multline*}
			(\eta_1-\eta)\kappa_{\text{fcd}} \sum_{k\in \mathcal{S}\setminus \mathcal{S}(K_{\eta})} h_k \min\left(\frac{h_k}{1+\omega_k}, \Delta_k\right)  \le
			\sum_{k\in \mathcal{S}\setminus \mathcal{S}(K_{\eta})} \ared_k \\
			\le \sum_{k\in \mathcal{S}\setminus \mathcal{S}(K_{\eta})} f(u_k) + \phi_k(u_k) - (f(u\kp) +  \phi\kp(u\kp)) \\ 
			\le \sum_{k\in \mathcal{S}\setminus \mathcal{S}(K_{\eta})} f(u_k) + \phi_k(u_k) - (f(u_{s(k)}) +  \phi_{s(k)}(u_{s(k)}))  , 
		\end{multline*} 
		where we define $s(k)$ to be the index of the next successful iterate after iteration $k$, use the fact that $u_k$ is only updated in successful steps and again monotonicity of $\epsilon$. The expression on the right hand side is finite as it is a finite telescoping sum by boundedness of $F$.
		Thus, one can proceed as in the remaining part of the proof of \cite[Theorem 3]{240proxTrustRegion} to obtain the result. Here, we use that one can also transfer \cite[Lemma 8]{240proxTrustRegion} and \cite[Lemma 9]{240proxTrustRegion} to our setting.\\
		Note that for concluding $\liminf_{k\to\infty} h(u_k)=0$ from $\liminf_{k\to\infty} h_k=0$, we make use of the convexity of the upper bound $\phi_k$. 
	\end{proof}
	
	Note that in the previous proof it was shown that the sequence $f(u_k) + \phi_k(u_k)$ is monotonically decreasing and for successful iterates it holds
	\[
	(f(u\kp) +  \phi_{\kp}( u\kp)) \le f(u_k) + \phi_k(u_k) - (\eta_1-\eta)\kappa_{\text{fcd}} h_k \min\left(\frac{h_k}{1+\omega_k}, \Delta_k\right) . 
	\]

	Moreover, one can observe that under the additional assumption of bounded curvature, the whole sequence $h_k$ converges to zero. 
	
	Using similar arguments as in Theorem \ref{tm:convHk}, one can also transfer the complexity analysis from \cite{240proxTrustRegion} to the setting considered here.
	
	\begin{theorem}
		Let the requirements of Theorem \ref{tm:convHk} hold. Furthermore, assume that there exists $\kappa_{\text{curv}}>1$ independent of $k$ such that $\omega_k \le \kappa_{\text{curv}}-1$ for all $k$, i.e. the model curvature is uniformly bounded. Let $\delta \in (0,1]$. Then after at most $\mathcal{O}(\delta^{-2})$ iterations, Algorithm \ref{alg:smoothingTrForLp} satisfies $h_k \le \delta$.
	\end{theorem}
	\begin{proof}
		Using the same arguments as in the proof of Theorem \ref{tm:convHk}, it follows for successful iterates that
		\[
		\ared_k^{(\epsilon)} \le f(u_k) + \phi_k(u_k) - f(u\kp) - \phi\kp(u\kp).
		\]
		Thus we can modify the first inequality in the proof of \cite[Lemma 10]{240proxTrustRegion} to
		\[
		f(u_k) + \phi_k(u_k) - f(u\kp) - \phi\kp(u\kp) \ge \frac{\kappa_{\text{fcd}}}{\kappa_{\text{curv}}} (\eta_1-\eta)\min(1, \gamma_1^{K_{\eta}}\Delta_1, \gamma_1\kappa_{\text{vs}}) \delta^2,
		\]
		which allows to proceed as in the remaining part of the proofs of \cite[Lemma 10, Lemma 11, Theorem 4]{240proxTrustRegion} to obtain the statement.
	\end{proof}
	
	\subsubsection{\texorpdfstring{$\Ltwo$}{L2}-case}
	Next, we show that under some assumptions, the iterates of the algorithm satisfy the optimality conditions from section \ref{sec:optConds}. We start with the $\Ltwo$-setting.
	\begin{lemma}
		Let the assumptions of Theorem \ref{tm:convHk} be satisfied and assume there are only finitely many successful iterations.
		Then the iterates $u_k$ are all equal to $\bar u$ for $k$ large enough, and $\bar u$ satisfies the necessary optimality conditions 
		\[
		-\nabla f(\bar u) +\alpha \bar u \in \partial j(\bar u) + \partial\Iuad(\bar u) \quad \text{and} \quad -\nabla f(\bar u) +\alpha \bar u \in \hat \partial j(\bar u) + \hat \partial\Iuad(\bar u).
		\]
	\end{lemma}
	\begin{proof}
		By definition of Algorithm \ref{alg:smoothingTrForLp}, the iterates $u_k$ stay the same for $k$ large enough, i.e. $u_k=\bar u\coloneq u_{K_s}$ for all $k>K_s$, where $K_s$ denotes the index of the last successful iterate. Furthermore, the trust region update rule implies $\Delta_k\to 0$ in case of only finitely many successful steps. Let $\delta>0$. Due to $\epsilon_k \to 0$, for $k$ large enough it holds $\epsilon_k<\delta$ and on $\{|\bar u| >\delta\}$  the integrand in $j_k(u)$ is therefore independent of $\epsilon_k$ and equals $|\bar u|^p + \frac p2|\bar u|^{p-2}(u^2-\bar u^2)$. 
		Thus, $\prox_{r_0\phi_k}(\bar u-r_0 f(\bar u))-\bar u $ stays constant on $\{|\bar u| >\delta\}$ for $k$ large enough and has to be zero by Theorem \ref{tm:convHk}.  
		From \eqref{eq:formulaProxPhik} one thus obtains 
		\[
		\bar u = \prox_{r_0\phi_k}(\bar u-r_0 \nabla f(\bar u)) = \proj_{[a,b]}\left(\frac{\bar u-r_0 \nabla f(\bar u)}{1+\alpha r_0 + 2r_0\beta \psi_{\epsilon}'(\bar u^2)}\right).
		\]
		Rearranging this and using $ \psi_{\epsilon}'(\bar u^2)= \frac{p}{2}|\bar u|^{p-2}$ yields
		\begin{equation}\label{eq:proofSubdiffInclusion}
			-\nabla f(\bar u)-\alpha \bar u \,\, \begin{cases}
				\le  \,\,\,\, \beta p |\bar u|^{p-2}\bar u  \qquad &\text{ a.e. on } \{ \bar u = a \}, \\
				= \,\,\,\, \beta p |\bar u|^{p-2}\bar u \qquad &\text{ a.e. on } \{ |\bar u| > \delta \} \cap \{a<\bar u<b\}, \\
				\ge   \,\,\,\, \beta p |\bar u|^{p-2}\bar u  \qquad &\text{ a.e. on } \{ \bar u = b \}. 
			\end{cases} 
		\end{equation}
		As $\delta >0$ was arbitrary, this shows $-\nabla f(\bar u) - \alpha \bar u \in \partial \phi(\bar u) = \hat\partial\phi(\bar u)$ by Lemma \ref{lm:LpSubdiffIndicator}.
	\end{proof}

	\begin{lemma}
		Let $u_k\to \bar u$ along a subsequence with $\lim_{k\to\infty} h(u_k) = 0$. 
		Then the limit $\bar u$ satisfies the necessary optimality conditions 
        \[
        -\nabla f(\bar u) + \alpha \bar u \in \partial j(\bar u) + \partial \Iuad(\bar u)  \qquad \text{and} -\nabla f(\bar u) + \alpha \bar u \in \hat\partial j(\bar u) + \hat\partial \Iuad(\bar u) .\]
	\end{lemma}
	\begin{proof}
		By assumption it holds $\prox_{r_0\phi_k}(u_k-r_0 \nabla f(u_k)) \to \bar u$ (along that subsequence). By \eqref{eq:formulaProxPhik}, it thus holds 
		\[
		\proj_{[a,b]}\left(\frac{u_k-r_0 \nabla f(u_k)}{1+\alpha r_0 + 2\beta r_0 \psi_{\epsilon_k}'(u_k^2)} \right) \to \bar u
		\]
		in $\Ltwo$. Let $\delta >0$. Then we can extract a subsequence such that 
		\[
		\proj_{[a,b]}\left( \frac{u_k-r_0 \nabla f(u_k)}{1+\alpha r_0 + 2\beta r_0 \psi_{\epsilon_k}'(u_k^2)} \right) \to \proj_{[a,b]}\left(\frac{\bar u-r_0 \nabla f(\bar u)}{1+\alpha r_0 + \beta r_0 p|\bar u|^{p-2}} \right)  
		\]
		pointwise a.e. on $\{|u|>\delta\}$.
		Combining these two results leads to \eqref{eq:proofSubdiffInclusion} as in the previous Corollary. As $\delta>0$ was arbitrary, the claim follows by Lemma \ref{lm:LpSubdiffIndicator}.
	\end{proof}
	
	Recall that under appropriate assumptions, a sequence with $\lim_{k\to\infty} h(u_k) = 0$ exists by Theorem \ref{tm:convHk}. 
	
	\begin{remark}
		Limits $\bar u$ as in the previous lemmas that satisfy $h_k(u_k)\to 0$ with $u_k\to \bar u$ are not necessarily $L$-stationary. Indeed, if for example $\bar u=0$, then $\prox_{r_0\phi_k}(u_k-r_0\nabla f(u_k)) \to 0$ for $k\to \infty$ independently of $\nabla f(u_k)$ and $\nabla f(\bar u)$. This can be seen from passing to the limit $\epsilon_k\to 0$ in \eqref{eq:formulaProxPhik} pointwise a.e. on $\Omega$, as the numerator is converging pointwise a.e. (on a subsequence) and the denominator tends to infinity. However, for $0 \in \prox_{r_0\phi}(-r_0\nabla f(\bar u))$ the absolute value of $\nabla f(\bar u)$ has to be below a certain threshold $q_0$, see \cite[Lemma 3.6]{15proxNonsmooth}.
	\end{remark}
	
	\subsubsection{\texorpdfstring{$\Hone$}{H01}-case}
	Now, let $U=\Hone$. Moreover, let $p_k \coloneqq \prox_{r_0\phi_k}(u_k -r_0 \nabla f(u_k))$ and define $\lambda_k \in \Hone^*$ by $\braket{\lambda_k,v}_{\Hone} = 2 \io \psi_{\epsilon_k}'(u_k^2)p_k v \dd x$.
	We start with the following auxiliary result similar to \cite[Lemma 7.7]{sparse}.

	\begin{lemma}\label{lm:auxLambdaPkConv}
		Assume $u_k \rightharpoonup\bar u$ and $(p_k-u_k)\to 0$ in $\Hone$.
		Then it holds 
		\[
		2\io \psi_{\epsilon_k}'(u_k^2)p_k^2 \to p \io |\bar u|^p\dd x.
		\]
	\end{lemma}
	\begin{proof}
		The assumptions on the convergence of $u_k$ and $(p_k-u_k)$ together with uniqueness of weak limits imply $p_k\rightharpoonup \bar u$.
		As noted in \eqref{eq:H1prox}, $p_k$ satisfies
		\begin{equation*}
			\frac{1}{r_0} (p_k-u_k, v)_{\Hone} + \braket{\nabla f(u_k),v}_{\Hone} + 2\beta\io \psi_{\epsilon_k}'(u_k^2)p_k v \dd x + \alpha(p_k,v)_{\Hone} = 0 \qquad \forall v\in \Hone. 
		\end{equation*}
		As in \cite[Lemma 7.3]{sparse}, we test this equation with $(p_k - u_k)$. Using the identity $a(a-b)=\frac12(a^2-b^2+(a-b)^2)$  and concavity of $\psi_{\epsilon}$, we obtain
		\begin{multline*}
			0=\frac{1}{r_0}\|p_k-u_k\|_{\Hone}^2 + \braket{\nabla f(u_k), p_k-u_k}_{\Hone} \\ + 
			2\beta\io \psi_{\epsilon_k}'(u_k^2)p_k (p_k-u_k) \dd x + \alpha(p_k,p_k-u_k)_{\Hone} \\
			= \frac{1}{r_0}\|p_k-u_k\|_{\Hone}^2 + \braket{\nabla f(u_k), p_k-u_k}_{\Hone} \\ +
			\beta\io \psi_{\epsilon_k}'(u_k^2)(p_k^2 -u_k^2 + (p_k-u_k)^2) \dd x + \alpha(p_k,p_k-u_k)_{\Hone} \\
			\ge \frac{1}{r_0}\|p_k-u_k\|_{\Hone}^2 + \braket{\nabla f(u_k), p_k-u_k}_{\Hone} \\ +
			\beta\io \psi_{\epsilon_k}(p_k^2)- \psi_{\epsilon_k}(u_k^2) +\psi_{\epsilon_k}'(u_k^2) (p_k-u_k)^2 \dd x + \alpha(p_k,p_k-u_k)_{\Hone}.
		\end{multline*}
		By Assumption \ref{ass:fPhi} and boundedness of $(u_k)$, the sequence $(\nabla f(u_k))$ is bounded. Therefore, the assumptions imply that all the terms on the right hand side apart from $\io \psi_{\epsilon_k}'(u_k^2) (p_k-u_k)^2 \dd x$ are converging to zero. Here we have used the compact embedding of $\Hone$ into $\Ltwo$ and Lemma \ref{lm:psieps} for the part $\io \psi_{\epsilon_k}(p_k^2)- \psi_{\epsilon_k}(u_k^2) \dd x$. By nonnegativity of the remaining term, we thus obtain 
		\[
		\io \psi_{\epsilon_k}'(u_k^2)(p_k-u_k)^2 \dd x \to 0.
		\]
		This enables us to proceed as in \cite[Lemma 7.7]{sparse} and obtain 
		\[
		2\io \psi_{\epsilon_k}'(u_k^2)p_k^2 \to p \io |\bar u|^p\dd x.
		\]
	\end{proof}
	
	Analogously to the $\Ltwo$ case, we consider the case of only finitely many successful iterations in the trust-region algorithm.
	\begin{lemma}
		Let the assumptions of Theorem \ref{tm:convHk} be satisfied and assume there are only finitely many successful iterations.
		Then the iterates $u_k$ are all equal to $\bar u$ for $k$ large enough, and $\bar u$ satisfies the necessary optimality condition 
		\begin{equation}\label{eq:ocH1Part1}
			\alpha (\bar u,v)_{\Hone} + \beta \braket{\bar\lambda,v}_{\Hone} = -f'(\bar u) v \qquad \forall v\in U
		\end{equation}
		with 
		\begin{equation}\label{eq:ocH1FinPart2}
			\braket{\bar \lambda,\bar u} = p \io |\bar u|^p \dd x.
		\end{equation}
	\end{lemma}
	\begin{proof}
		If there are only finitely many successful iterations, it holds $u_k = \bar u$ for $k$ large enough.
		Theorem \ref{tm:convHk} yields $p_k = \prox_{r_0\phi_k}(\bar u-r_0 \nabla f(\bar u)) \to \bar u$ in $\Hone$ along a subsequence. By \eqref{eq:H1prox}, it holds
		\[
		\frac{1}{r_0} (p_k-\bar u, v)_{\Hone} + \braket{\nabla f(\bar u),v}_{\Hone} + 2\beta\io \psi_{\epsilon_k}'(\bar u^2)p_k v \dd x + \alpha(p_k,v)_{\Hone} = 0 \qquad \forall v\in \Hone. 
		\]
		Convergence of the remaining terms implies $\lambda_k\to \bar \lambda$ in $\Hone^*$ for some $\bar\lambda$ in $\Hone^*$. Passing to the limit in the previous equation yields \eqref{eq:ocH1Part1}. The second part \eqref{eq:ocH1FinPart2} follows from Lemma \ref{lm:auxLambdaPkConv} and $\braket{\lambda_k,p_k}_{\Hone} \to \braket{\bar\lambda,\bar u}_{\Hone}$ by strong convergence of $p_k$ and $\lambda_k$.
	\end{proof}

	\begin{lemma}
		Assume that $\nabla f$ is completely continuous. Let $\bar u$ be a weak accumulation point of a subsequence satisfying $h_k(u_k) \to 0$. Then $\bar u$ satisfies the necessary optimality condition 
		\begin{equation}\label{eq:ocH1WeakPart1}
			\alpha (\bar u,v)_{\Hone} + \beta \braket{\bar\lambda,v}_{\Hone} = -f'(\bar u) v \qquad \forall v\in U
		\end{equation}
		with 
		\begin{equation}\label{eq:ocH1WeakPart2}
			\braket{\bar \lambda,\bar u} \ge p \io |\bar u|^p \dd x.
		\end{equation}
	\end{lemma}
	\begin{proof}
		Let us choose a subsequence satisfying the assumption, still denoted by $(u_k)$, with $u_k \rightharpoonup \bar u$ in $\Hone$. Then it also holds $p_k \coloneqq \prox_{r_0\phi_k}(u_k-r_0 \nabla f(u_k)) \rightharpoonup \bar u$ in $\Hone$. As noted in \eqref{eq:H1prox}, $p_k$ satisfies
		\begin{equation*}
			\frac{1}{r_0} (p_k-u_k, v)_{\Hone} + \braket{\nabla f(u_k),v}_{\Hone} + 2\beta\io \psi_{\epsilon_k}'(u_k^2)p_k v \dd x + \alpha(p_k,v)_{\Hone} = 0 \quad \forall v\in \Hone. 
		\end{equation*}
		Since all the remaining terms are converging (weakly) by complete continuity of $\nabla f$, there exists $\bar \lambda$ such that $\lambda_k  \rightharpoonup \bar \lambda$ in $\Hone^*$.
		Testing the equation with $p_k$  as in the proof of \cite[Theorem 7.8]{sparse} yields 
		\[
		\frac{1}{r_0} (p_k-u_k, p_k)_{\Hone} + \braket{\nabla f(u_k),p_k}_{\Hone} + \alpha(p_k,p_k)_{\Hone} = - \beta \braket{\lambda_k,p_k}_{\Hone}.
		\]
		Passing to the limit inferior on the left hand side and to the limit on the right hand side using Lemma \ref{lm:auxLambdaPkConv} leads to
		\[
		\braket{\nabla f(\bar u),\bar u}_{\Hone} + \alpha\|\bar u\|_{\Hone}^2 \le  - \beta p \io |\bar u |^p\dd x.
		\]
		By \eqref{eq:ocH1WeakPart1}, the left hand side equals $-\beta \braket{\bar \lambda,\bar u}_{\Hone}$, which implies \eqref{eq:ocH1WeakPart2}.
	\end{proof}
	Note that as we are only working with weakly convergent subsequence of $(u_k)$, we obtain only an inequality in \eqref{eq:ocH1WeakPart2}. Strong accumulation points satisfy the stronger optimality condition from Theorem \ref{tm:convHk}.
	
	\section{Generalized Cauchy points}\label{sec:GCP}
		To compute trial steps in step 4 of Algorithm \ref{alg:smoothingTrForLp} that satisfy the FCD condition \eqref{eq:fcdCond}, we follow a proximal gradient path for the convex model $m_k$ to obtain a generalized Cauchy point as done in \cite{240proxTrustRegion}.
	We abbreviate the proximal gradient paths for the convex and nonconvex case by 
	\[
	p_k(r) \coloneq \prox_{r\phi_k}(u_k - r g_k)-u_k \qquad \text{and} \qquad p_{k}^{\epsilon}(r) \coloneq \prox_{r\phi_{\epsilon}}(u_k - r g_k)-u_k
	\]
	and the associated trial iterates by
	\[
	u_k(r) \coloneq  \prox_{r\phi_k}(u_k - r g_k) \qquad \text{and} \qquad  u_{k}^{\epsilon}(r) \coloneq  \prox_{r\phi_{k}^{\epsilon}}(u_k - r g_k).
	\]
	Furthermore, let 
	\[
	Q_k^{(\epsilon)}(r) \coloneqq  \braket{g_k, p_k^{(\epsilon)}(r)} + \phi_{k/\epsilon_k}(u_k^{\epsilon}(r)) - \phi_{k/\epsilon_k}(u_k^{(\epsilon)}).
	\]
	A generalized Cauchy point (GCP) or a nonconvex generalized Cauchy point (NCGCP) is a point $u_k(t_k^A)$ or $u_k^{\epsilon}(t_k^A)$ for which the step length $t_k^A$ satisfies 
	\begin{subequations}\label{eq:gcpCond}
		\begin{align} 
			m_k^{(\epsilon)}(u_k^{(\epsilon)}(t_k^A))-m_k^{(\epsilon)}(u_k) &\le \mu_1 Q_k^{(\epsilon)}(t_k^A), \label{eq:cpDesc} \\
			\|p_k^{(\epsilon)}(t_k^A)\|&\le \nu_1 \Delta_k \label{eq:cpTr}
		\end{align}
		and 
		\begin{equation} \label{eq:cpLowerCondTka}
			t_k^A \ge \nu_2 t_k^B \qquad \text{ or } \qquad
			t_k^A \ge \nu_3. 
		\end{equation}
		Here, $t_k^B$ satisfies 
		\begin{equation}
			m_k^{(\epsilon)}(u_k^{(\epsilon)}(t_k^B))-m_k^{(\epsilon)}(u_k) \ge \mu_2 Q_k^{(\epsilon)}(t_k^B)
		\end{equation} 
		or 
		\begin{equation}
			\|p_k^{(\epsilon)}(t_k^B)\| \ge \nu_4 \Delta_k.
		\end{equation}
	\end{subequations}
	
	The parameters are chosen such that
	\begin{align*}
		0 <\mu_1&<\mu_2 < 1, \\
		0<\nu_4<&\nu_1  \le \kappa_{\text{rad}}, \\
		0<\nu_2<1 \qquad &\text{and} \qquad 0<\nu_3.
	\end{align*}
	
	We compute such a generalized Cauchy point via a bidirectional line search along the generalized Cauchy path as described in \cite[Algorithm 2]{240proxTrustRegion}. The latter algorithm is restated here as Algorithm \ref{alg:genCP}. 
	
	\begin{algorithm} 
		\caption{Generalized Cauchy Point Algorithm}\label{alg:genCP}
		\begin{algorithmic}[1]
			\Require Previous step length $t_{k-1}^A>0$ if $k\ge1$ or initial step length $t_0^A>0$ if $k=0$, parameters $0<\beta_{\text{dec}}< 1 < \beta_{\text{inc}}$ and $M_{\text{inc}}\in \N$
			\If {conditions \eqref{eq:cpDesc} and \eqref{eq:cpTr} are satisfied}
			\State Set $M_k^{\text{inc}} = \max(M_{\text{inc}}, \lceil \log_{\beta_{\text{inc}}}(t_0^A/t_{k-1}^A) \rceil )$
			\State Compute the largest $l\in \{0,..., M_k^{\text{inc}}\}$ such that $t_k^A=t_{k-1}^A\beta_{\text{inc}}^l$ satisfies  \eqref{eq:cpDesc} and \eqref{eq:cpTr}
			\Else 
			\State Compute the smallest $l \in \N$ such that $t_k^A = t_{k-1}^A\beta_{\text{dec}}^l$ satisfies  \eqref{eq:cpDesc} and \eqref{eq:cpTr}
			\EndIf
		\end{algorithmic}
	\end{algorithm}
	
	When working with the convex model $m_k$, one can directly use results from \cite{240proxTrustRegion}.
	\begin{corollary}\label{cor:gcpAlgoTerminates}
		Algorithm \ref{alg:genCP} terminates in finitely many iterations with an iterate satisfying the FCD condition \eqref{eq:fcdCond} with model $m_k$.
	\end{corollary}
	\begin{proof}
		This combines \cite[Corollary 1]{240proxTrustRegion} and \cite[Theorem 2]{240proxTrustRegion}.
	\end{proof}
	
	We remark that it would be sufficient to compute a trial iterate that reduces the predicted reduction by at least a fraction of the one obtained by the Cauchy point as described in \cite[section 3.1]{240proxTrustRegion}. \\
	Equations \eqref{eq:formulaProxPhik} and \eqref{eq:H1prox} provide an efficient way to compute proximal points for the convex case.
	With Algorithm \ref{alg:genCP}, we can therefore compute a trial step that satisfies the FCD condition \eqref{eq:fcdCond} efficiently for the convex model.\\ 
	Next, we investigate the use of Algorithm \ref{alg:genCP} to compute NCGCPs in the $\Ltwo$-setting.

	We first prove the following auxiliary result.
	
	\begin{lemma} \label{lm:qkInequ}
		The function $r\mapsto Q_{k}^{\epsilon}(r)$ is nonincreasing. Furthermore, it holds 
		\begin{equation}\label{eq:qkEstimate}
			Q_{k}^{\epsilon}(r) = (g_k, p_{k}^{\epsilon}(r) ) + \phi_{\epsilon}(u_{k}^{\epsilon}(r)) -\phi_{\epsilon}(u_k) \le - \frac12 \Phi_{\epsilon}(r) \Psi_{\epsilon}(r). 
		\end{equation}
	\end{lemma}
	Here, $\Phi_{\epsilon}$ and $\Psi_{\epsilon}$ are defined as in Lemma \ref{lm:phiPsi} with the function $\phi_{\epsilon}$.
	\begin{proof}
		We proceed as in the proof of \cite[Lemma 7]{240proxTrustRegion}. To prove that $Q_{k}^{\epsilon}(r)$ is nonincreasing, let $r,t>0$. Using Lemma \ref{lm:lm1proxIneq} with $u=u_k-r g_k$ and $v = u_{k}^{\epsilon}(t)$, it holds by definition of $u_{k}^{\epsilon}(r)$ that
		\[
		(u_{k}^{\epsilon}(r)-(u_k-rg_k), u_{k}^{\epsilon}(t)-u_{k}^{\epsilon}(r)) + \frac{1}{2} \| u_{k}^{\epsilon}(r) - u_{k}^{\epsilon}(t)\|^2 \ge r \phi_{\epsilon}(u_{k}^{\epsilon}(r)) - r \phi_{\epsilon}(u_{k}^{\epsilon}(t))
		\]
		and with the roles of $r$ and $t$ reversed that
		\[
		(u_{k}^{\epsilon}(t)-(u_k-tg_k), u_{k}^{\epsilon}(r)-u_{k}^{\epsilon}(t)) + \frac{1}{2} \| u_{k}^{\epsilon}(t) - u_{k}^{\epsilon}(r)\|^2 \ge t \phi_{\epsilon}(u_{k}^{\epsilon}(t)) - t \phi_{\epsilon}(u_{k}^{\epsilon}(r)).
		\]
		Adding these two inequalities yields
		\[
		0 \le (t-r) \left[ (u_{k}^{\epsilon}(r)-u_{k}^{\epsilon}(t), g_k) + \phi_{\epsilon}(u_{k}^{\epsilon}(r)-\phi_{\epsilon}(u_{k}^{\epsilon}(t))) \right] = (t-r) (Q_{k}^{\epsilon}(r) - Q_{k}^{\epsilon}(t) ),
		\]
		which shows the claim. \\
		Inequality \eqref{eq:qkEstimate} follows from part 2 of Lemma \ref{lm:phiPsi} with $d=-g_k$.
	\end{proof}
	
	Now we are able to generalize Corollary \ref{cor:gcpAlgoTerminates} to the nonconvex $L^p$-setting.
	\begin{lemma}\label{lm:fcdResult}
		Let $U=\Ltwo$. For the nonconvex function $\phi_{\epsilon}$ and also for the non-smoothed function $\phi$ defined in \eqref{eq:defPhi}, Algorithm \ref{alg:genCP} terminates in finitely many iterations with a NCGCP.
	\end{lemma}
	\begin{proof}
        One can observe that  $\prox_{t\phi_{\epsilon}}(u+td) - u \to 0$ pointwise and monotonically, since for $t\searrow 0$ the first term with factor $\frac1{2t}$ dominates the pointwise proximal minimization problem 
		\[
		\arg\min_v \frac1{2t}(v-u)^2 - (v-u) d + \frac{\alpha}{2}v^2 + \beta \psi_{\epsilon}(v^2).
		\]
		This implies $\prox_{t\phi_{\epsilon}}(u+td) - u \to 0$ by the monotone convergence theorem. 
        Thus, there is $t_0>0$ such that \eqref{eq:cpTr} is satisfied for all $t < t_0$. The existence of $t_1>0$ such that \eqref{eq:cpDesc} is satisfied for all $t<t_1$ follows as in the convex case in \cite[Theorem 1]{240proxTrustRegion} by the fundamental theorem of calculus and Lemma \ref{lm:qkInequ}.
		However, here we do not treat the case $p_k(t)=0$ separately, because in Lemma \ref{lm:qkInequ} we do not assume $h_k>0$ as in \cite[Lemma 7]{240proxTrustRegion}. 
		With these observations one can proceed as in \cite[Corollary 1]{240proxTrustRegion} to obtain the result. 
	\end{proof}
	\begin{remark}
		For general nonconvex functions $\phi$ it does not hold $\lim_{r\to 0} \prox_{r\phi}(u+rd) \to u$ as in the convex case. A counterexample is for example $\phi(u) = I_{\Z}(u)$ at some $u\notin \Z$.
	\end{remark}
	
	In the convex case, one can prove that generalized Cauchy points satisfy the FCD condition for an appropriately chosen constant $\kappa_{\text{fcd}}$, see Corollary \ref{cor:gcpAlgoTerminates} and \cite[Theorem 2]{240proxTrustRegion}. However, we were not able to obtain such a result in the nonconvex case with the corresponding nonconvex version of $h_k$. 
	Since the function $\Psi$ from Lemma \ref{lm:phiPsi} is not nonincreasing as in the convex case it is harder to relate the decrease of the NCGCP from inequality \eqref{eq:cpDesc} to the nonconvex version of $h_k$. 
	However, with the help of the following lemma one can prove that a NCGCP satisfies the FCD condition \eqref{eq:fcdCond} with the nonconvex model $m_k^{\epsilon}$ instead of the convex one $m_k$ in the unconstrained $\Ltwo$-setting. Note that in the FCD condition \eqref{eq:fcdCond} with nonconvex model we still work with the convex upper bound $\phi_k$ in the definition of $h_k$. 
	
	\begin{lemma}\label{lm:proxDistInequ}
		Let $U=\Ltwo=\Uad$.
		Then there is a constant $c>0$ such that 
		\begin{equation}\label{eq:proxDistIneq}
			\| \prox_{r\phi_k}(u_k - r g_k) -u_k\|_{\Ltwo} \le c  \| \prox_{r\phi_{\epsilon}}(u_k - r g_k) -u_k\|_{\Ltwo}. 
		\end{equation}
	\end{lemma}
	\begin{proof}
		Define $p_1 \coloneq \prox_{r\phi_{\epsilon}}(u_k - r g_k)$ and $p_2 \coloneq \prox_{r\phi_k}(u_k - r g_k) $.
		Let the objective functions of those pointwise proximal minimization problems be denoted by 
		\begin{align*}
			p_{\epsilon}(u) &\coloneqq  \frac{1}{2r}(u-(u_k-rg_k))^2 + \frac{\alpha}{2}u^2 + \beta \psi_{\epsilon_k}(u^2) \qquad \text{and}  \\
			p_{c}(u) &\coloneqq \frac{1}{2r}(u-(u_k-rg_k))^2 + \frac{\alpha}{2}u^2 + \beta \left( \psi_{\epsilon_k}(u_k^2) + \psi_{\epsilon_k}'(u_k^2)(u^2-u_k^2)\right)
		\end{align*}
		for some fixed iterate $k\ge 0$.
		Let their difference be given by 
		\[
		h(u) \coloneqq p_c(u) - p_{\epsilon}(u) = \beta  \left( \psi_{\epsilon_k}(u_k^2) + \psi_{\epsilon_k}'(u_k^2)(u^2-u_k^2) - \psi_{\epsilon_k}(u^2)\right).
		\]
		Its derivative reads 
		\begin{equation} \label{eq:hPrime}
			h'(u) = 2\beta u \left(\psi_{\epsilon_k}'(u_k^2) - \psi_{\epsilon_k}'(u^2)\right).
		\end{equation}
		Note that it holds $\sign(\prox_{r\phi_{k/\epsilon}}(u_k - r g_k)) = \sign(u_k-rg_k)$. This follows from the fact that the terms in $p_{\epsilon}$ and $p_c$ apart from the first one have their minimum at zero. 
		We prove inequality \eqref{eq:proxDistIneq} pointwise by exhaustion.\\
		\textbf{Case 1: $u_k \ge 0, u_k-rg_k \ge 0$.} 
		\textbf{Case 1.1: $u_k \le  p_1$.} Then it holds by concavity that $\psi_{\epsilon_k}'(u_k^2) \ge \psi_{\epsilon_k}'(p_1^2)$ 
		and therefore $0\le h'(p_1)  = p_c'(p_1)$ by \eqref{eq:hPrime}, definition of $h$ and $p_{\epsilon}'(p_1)=0$. As $p_c$ is an upwards-opening parabola and $p_c'(p_2)=0$, this implies $p_2 \le p_1$. \\
		If $p_2 \ge u_k$, this directly proves inequality \eqref{eq:proxDistIneq}. 
		Thus, we now assume $p_2 < u_k$.\\
		Let $\epsilon_k < u_k$. We compute $p_c''(u_k) = \frac1r + \alpha + \beta p u_k^{p-2}$ and 
		\begin{equation}\label{eq:secondDerPEps}
			p_{\epsilon}''(v) = \frac1r + \alpha + \beta p (p-1) v^{p-2}
		\end{equation} 
		for $v>\epsilon_k$, so $p_c''(u_k) \ge |p_{\epsilon}''(u_k)|$ and $p_{\epsilon}''$ is monotonically increasing by \eqref{eq:secondDerPEps} and $p\in(0,1)$.  Note that since $0<p_c'(u_k)=p_{\epsilon}'(u_k)$ and $p_{\epsilon}'(p_1) = 0$, it holds $p_{\epsilon}''(u_k)<0$. Therefore, $|p_{\epsilon}''|$ is smaller than the constant value of $p_c'$ at least until the following local maximum $u_m$ that precedes the inflection point with $p_{\epsilon}''(v)=0$. Thus, the derivative $p_{\epsilon}'$ decreases slower on $(u_k, u_m)$ than $p_c'$ increases on $(p_2, u_k)$. By $p_c'(u_k) = p_{\epsilon}'(u_k)$, this implies that $|p_2-u_k|\le |u_k-u_m| \le |u_k - p_1|$. If $u_k\le \epsilon_k$, one can follow the same argumentation with $\epsilon_k$ instead of $u_k$, replacing the undefined value of $p_{\epsilon}''(\epsilon_k)$ with $\lim_{v\searrow\epsilon_k} p_{\epsilon}''(v)$.  \\ 
		\textbf{Case 1.2: $u_k \ge p_1$.} In this case, it holds $\psi_{\epsilon_k}'(u_k^2) \le \psi_{\epsilon_k}'(p_1^2)$, so $0 \ge h'(p_1) = p_c'(p_1)$. Then $p_2 \ge p_1$ because $p_c$ is an upwards-opening parabola and the inequality follows directly if $p_2 \le u_k$. 
		Otherwise, if $p_2 \ge u_k$, it holds $h'(p_2) \ge 0$, so $p_{\epsilon_k}'(p_2) = p_c'(p_2) - h'(p_2) = -h'(p_2) \le 0$. Since w.l.o.g. $p_1<p_2$, $p_1$ has to be the first of the two local minima of $p_{\epsilon}$ and thus lies below the threshold of Lemma \ref{lm:sparsityProxLp} with $p_1\in [0, \epsilon_k]$. \\
		If $u_k < \epsilon_k$, then  $p_{\epsilon}$ is already a quadratic function at $u_k$ and thus $p_2 = p_1$ and the claimed inequality holds. \\
		Therefore, assume now $\epsilon_k \le u_k \le 2 \epsilon_k$. 
		As in the previous case, we have $p_c''(v) = \frac1r + \alpha + \beta p u_k^{p-2}$ for all $v\in \R$ and $p_{\epsilon}''(v) = \frac1r + \alpha + \beta p (p-1) v^{p-2}$ for all $v > \epsilon$.
		This yields $|p_{\epsilon}''(u_k)|\le p_c''(u_k)$.
		Using $u_k\le 2\epsilon_k$, one can estimate
		\begin{multline*}
			\lim_{v\to \epsilon_k}  |p_{\epsilon}''(v)| \le \frac1r + \alpha + \beta p |p-1|\epsilon_k^{p-2} \le \frac1r + \alpha + \beta p u_k^{p-2} 2^{2-p} \\
			\le 2^{2-p} \left(\frac1r + \alpha + \beta p u_k^{p-2}\right) = 2^{2-p}p_c''(u_k).
		\end{multline*}
		Monotonicity of $p_{\epsilon}''$ then yields
		\[
		\sup_{v\in(\epsilon_k,u_k]}|p_{\epsilon}''(v)| = \max\left(\lim_{v\to \epsilon_k}  |p_{\epsilon}''(v)| , |p_{\epsilon}''(u_k)| \right) \le 2^{2-p} p_c''(u_k).
		\]
		Let $u_m$ be the local maximum on $(\epsilon_k, u_k)$. Since $p_{\epsilon}'(u_k)=p_c'(u_k)$, we obtain
		\[
		p_c''(u_k)|p_2-u_k|\le \sup_{v\in(\epsilon_k,u_k]}|p_{\epsilon}''(v)| |u_m-u_k| \le 2^{2-p} p_c''(u_k) |u_m-u_k| \le 2^{2-p} p_c''(u_k) |p_1-u_k|, 
		\]
		where we have used $p_1\le\epsilon_k<u_m$.
		Dividing by $p_c''(u_k)\ge 0$ yields inequality \eqref{eq:proxDistIneq}.\\
		
		It remains to consider the case $u_k \ge 2\epsilon_k$. We first establish some bound on $g_k$. Since $j_{\epsilon_k}$ majorizes $j$, the function $p_{\epsilon}$ also majorizes its nonsmoothed version $p_0$. By Lemma \ref{lm:sparsityProxLp} and $p_1 \le \epsilon$ as deduced at the beginning of case 1.2, the minimum of $p_0$ is zero and thus we obtain by optimality of $p_1$ that
		\begin{equation}\label{eq:auxIneqProxDist}
			\frac1{2r}(u_k-rg_k)^2 =p_0(0)\le p_{\epsilon}(p_1) \le p_{\epsilon}(u_k) = \frac1{2r} (rg_k)^2 + \frac{\alpha}{2}u_k^2 + \beta |u_k|^p.
		\end{equation}
		Note that $g_k \le 0$ in this case, as $u_k\le p_2\le u_k-rg_k$. For the second inequality we have used that the $u_k-rg_k$ minimizes the first term in $p_c$ and zero minimizes the remaining terms.  Rearranging \eqref{eq:auxIneqProxDist} and dividing by $u_k \ge 0 $ yields 
		\begin{equation} \label{eq:gkBound}
			-g_k \le \frac{\alpha}{2} u_k + \beta u_k^{p-1}-\frac{1}{2r} u_k .
		\end{equation}
		
		We can then estimate
		\begin{multline}\label{eq:distP2Uk}
			p_2-u_k = \frac{u_k - rg_k}{1+\alpha r + 2r\beta \psi_{\epsilon}'(u_k^2)} -u_k = \frac{ - rg_k - \alpha r u_k - 2r\beta u_k \psi_{\epsilon_k}'(u_k^2)}{1+\alpha r + 2r\beta \psi_{\epsilon}'(u_k^2)} \\
			\le  \frac{ r(\frac{\alpha}{2} u_k + \beta u_k^{p-1}-\frac{1}{2r} u_k)}{1+\alpha r + 2r\beta \psi_{\epsilon}'(u_k^2)} 
			\le \frac{\alpha r u_k}{2\alpha r} + \frac{r\beta u_k^{p-1}  }{pr \beta u_k^{p-2}} \le \frac{p+2}{2p} u_k. 
		\end{multline}
		
		Using $u_k\ge 2 \epsilon_k$, it holds
		\[
		u_k - p_1 \ge u_k - \epsilon_k \ge u_k-\frac12 u_k = \frac12 u_k.
		\] 
		Combining this with \eqref{eq:distP2Uk} leads to inequality \eqref{eq:proxDistIneq} with
		\[
		|p_2-u_k|\le \frac{p+2}{2p} u_k \le \frac{p+2}{p} |p_1-u_k|.
		\]
		
		\textbf{Case 2: $u_k \le 0, u_k - rg_k \ge 0$.} 
		\textbf{Case 2.1: $p_1 \ge |u_k|$.} By \eqref{eq:hPrime}, it holds $h'(p_1)\ge 0$, so $p_c'(p_1) = h'(p_1) + p_{\epsilon}'(p_1 ) = h'(p_1) \ge 0$. This yields $p_2\le p_1$, and thus inequality \eqref{eq:proxDistIneq} is satisfied.\\
		\textbf{Case 2.2:  $p_1 \le |u_k|$.} 
		If $p_2\le |u_k|$, then $|p_2-u_k|\le 2 |u_k| \le 2|p_1-u_k|$. Thus, assume $p_2 \ge |u_k|$, so $p_{\epsilon}'(-u_k)=p_c'(-u_k)\le 0$. As the minimum $p_1$ of $p_{\epsilon}$ is smaller than $|u_k|$, $p_{\epsilon}$ has a local maximum less than or equal to $-u_k$ and it has to hold $p_1\in [0,\epsilon_k]$ by Lemma \ref{lm:sparsityProxLp}.
		If $u_k < \epsilon_k$, then $p_c$ and $p_{\epsilon}$ again coincide on $(-\epsilon_k,\epsilon_k)$ and $p_1=p_2$.
		Therefore, assume now $u_k \ge \epsilon_k$.
		Similarly to case 1.2, it holds $p_0(0) \le p_{\epsilon}(-u_k)$, so 
		\[
		\frac{1}{2r} u_k^2 - u_k g_k + \frac{1}{2r}(rg_k)^2 \le \frac2r u_k^2 - 2 u_k g_k + \frac{1}{2r}(rg_k)^2 + \frac{\alpha}{2}u_k^2 + \beta |u_k|^p.
		\]
		Rearranging and dividing by $|u_k|$ yields
		\[
		-g_k \le \frac{3}{2r}|u_k| + \frac{\alpha}{2}|u_k| + \beta |u_k|^{p-1}.
		\]
		With \eqref{eq:formulaProxPhik}, we can then estimate
		\begin{multline*}
			p_2-u_k = \frac{-rg_k -\alpha r u_k -r\beta u_k p|u_k|^{p-2}}{1+\alpha r + r\beta p|u_k|^{p-2}} \\
			\le \frac{r \left( \frac{3}{2r}|u_k| + \frac{\alpha}{2}|u_k| + \beta |u_k|^{p-1}\right) -\alpha r u_k -r\beta u_k p|u_k|^{p-2}}{1+\alpha r + r\beta p|u_k|^{p-2}} 
			\le \left( \frac32 + \frac12 + \frac1p+1 +1\right) |u_k| 
		\end{multline*}
		Furthermore, it holds $p_1-u_k \ge |u_k| $. Combining these two estimates yields \eqref{eq:proxDistIneq} with $c=4 + \frac{1}{p}$.
		\textbf{Case 3: } $u_k \le 0, u_k-rg_k\le 0$ and \textbf{Case 4: } $u_k \ge 0, u_k-rg_k\le 0$ follow by symmetry from Cases 1 and 2. \\
	\end{proof}
	Note that from the previous proof one can deduce $c\le 4+\frac2p$ in inequality \eqref{eq:proxDistIneq}.
	\begin{remark}
		The presence of control constraints does not allow for an estimate as in the previous lemma. This can be seen for the case $0\le p_2 < u_k \le p_1$. If $u_k=b$, then $p_1=b=u_k$, but $|p_2-u_k|>0$.
	\end{remark}
	\begin{lemma}\label{lm:FCDnonconvexCase}
		Let $U=\Ltwo=\Uad$ and let $u_k^{\epsilon}(t_k^A)$ be a NCGCP. Then it satisfies the FCD condition \eqref{eq:fcdCond} with the nonconvex model $m_k^{\epsilon}$. 
	\end{lemma}
	Recall that $h_k$ in the FCD condition is defined via the convexified function $\phi_k$
	\begin{proof}
		With the help of inequality \eqref{eq:proxDistIneq} and Lemma \ref{lm:qkInequ} one obtains 
		\[
		-Q_k^{\epsilon}(r) \ge \frac12 \Phi_k^{\epsilon}(r) \Psi_k^{\epsilon}(r) \ge \frac 1{2c^2} \Phi_k(r) \Psi_k(r)
		\]
		This additional estimate allows to proceed as in the convex case for all the four Cases in \cite[Theorem 2]{240proxTrustRegion}, as the right hand side of the above inequality leads us back to the convex case. 
	\end{proof}
	
	\begin{corollary}\label{cor:gcpYieldsTrialNC}
		Let $U=\Ltwo=\Uad$. Using the nonconvex GCP conditions instead of \eqref{eq:gcpCond} in Algorithm \ref{alg:genCP} leads to a trial step that satisfies the FCD condition \eqref{eq:fcdCond} with the nonconvex model $m_k^{\epsilon}$.
	\end{corollary}
	We call this version NC-GCP algorithm.
	\begin{proof}
		This follows from combining Lemma \ref{lm:fcdResult} and Lemma \ref{lm:FCDnonconvexCase}. 
	\end{proof}
	The previous Corollary shows that we can also use the NC-GCP algorithm for the nonconvex $L^p$-case to obtain a trial step with sufficient decrease in Line 4 of Algorithm \ref{alg:smoothingTrForLp}, at least in the unconstrained $\Ltwo$-setting. 
	
	\section{Approximate subproblem solutions}\label{sec:subproblem}
	Having computed a generalized Cauchy point that satisfies the FCD condition \eqref{eq:fcdCond}, we want to improve the step by computing a better solution of the nonconvex subproblem. 
		To this end, we minimize the nonconvex model $m_k^{\epsilon}$ with an iterative monotonically decreasing algorithm starting from the iterate $u\kp$ computed in the first step with the GCP Algorithm \ref{alg:genCP}. This nonconvex step improvement leads to an iterate that satisfies the FCD condition \eqref{eq:fcdCond} with the nonconvex model $m_k^{\epsilon}$, see also Remark \ref{rm:fcdCondNonconv}. 

	\subsection{Nonconvex SPG}\label{sec:SPG}
	In \cite[Appendix C]{240proxTrustRegion}, a spectral proximal gradient (SPG) method was used as a subproblem solver to improve upon the trial steps computed with the GCP algorithm. We were trying to employ this algorithm also in our nonconvex setting, but there are two issues:\\
	Firstly, if $\phi$ is nonconvex, one can not argue as in the convex case that the line search of this SPG method terminates after finitely many steps. Thus, we iterate over the proximal parameter $r$ instead. This results in a proximal gradient method with variable step size and initial Barzilai-Borwein step size.\\
	Second, the SPG method requires the computation of a proximal point of 
	\[
	\phi_{\epsilon} + I_{\Delta_k}, \quad \text{where} \quad I_{\Delta_k} \coloneqq I_{B_{\Delta_k}(u_k)}.
	\]
	If $\phi_{\epsilon}$ was convex, this can be done by finding a zero of 
	\begin{equation}\label{eq:funForBrent}
		t\mapsto\|\prox_{rt\phi_{\epsilon}}(u_k + t(u-u_k)) - u_k\| - \Delta_k
	\end{equation}
	in case $\|\prox_{r\phi_{\epsilon}}(u)-u_k\|>\Delta_k$.
	We were not able to prove that roots of \eqref{eq:funForBrent} are proximal points of $\phi_{\epsilon}+I_{\Delta_k}$ as in the convex case. 
	As the function \eqref{eq:funForBrent} is not continuous in our nonconvex setting, it does not even necessarily have a root.
	Regardless, since the mapping \eqref{eq:funForBrent} is monotonically increasing by Lemma \ref{lm:phiPsi} one could compute the largest $t$ such that the function is negative and in this way obtain a point $\prox_{rt\phi_{\epsilon}}(u_k + t(u-u_k))$ that is inside the trust-region and might serve as an approximation for the proximal point of $\phi_{\epsilon}+I_{\Delta_k}$. Using this point as a new iterate of the SPG algorithm works in the examples in the last section, but we were not able to prove descent properties theoretically. 
	In \cite{240proxTrustRegion}, finding zeros of \eqref{eq:funForBrent} is accomplished using Brent’s method, a root-finding algorithm for continuous functions that combines bisection, quadratic interpolation and the secant method. Due to monotonicity of the function, we can also use Brent's method in our setting here to compute the largest $t$ such that the function \eqref{eq:funForBrent} is negative. 
	
	\subsection{Convexified SPG}\label{sec:MMSPG}
	Next, we introduce a convexified SPG algorithm to solve the nonconvex subproblem that is constructed in a way to use the line search from the SPG algorithm in \cite{240proxTrustRegion}. To this end, we construct a convex upper bound $ \phi_{k,l}$ in each iteration. The function $\phi_{k,l}$ is defined similarly as $\phi_k$ above, but for iterate $u_{k,l}$ and the objective of the subproblem.  Thus, we end up with a convex function and can proceed similarly as in the convex case.
	The resulting scheme is similar to the majorize-minimization (MM) method from \cite{sparse} and is given by Algorithm \ref{alg:MMSPG} (MM-SPG).
	
	\begin{algorithm}
		\caption{MM-SPG}\label{alg:MMSPG}
		\begin{algorithmic}[1]
			\Require Initial guess $u_{k,0} = u_k(t_k^A)$ with $f_{k,0}=f_k(u_{k,0})$, 
			$d_{k,0}=\nabla f_k(u_{k,0})$, $s=(u_{k,0}-u_k)$, $b=B_k s$. Parameters $\epsilon_1, \epsilon_2, \mu_1,\beta_1\in (0,1)$, $\lambda_{\min}<\lambda_{\max}$, $L_{\max} \in \N$.
			\State Set $l=0$
			\While {$l<L_{\max}$ \text{and} $h_{k,l}> \min(\epsilon_1, \epsilon_2 h_{k,0})$ }
			\If{$(b,s) \le 0$}
			\State Set $\lambda = 1/\|d_{k,l}\|$
			\Else
			\State Set $\lambda = (s,s)/(b,s)$
			\EndIf
			\State Set $\lambda_l = \max( \lambda_{\min}, \min(\lambda_{\max}, \lambda))$
			\State Define 
			\[
			\phi_{k,l}(u) = \beta \left(\io \psi_{\epsilon_k}(u_{k,l}^2)\dd x + \io \psi_{\epsilon_k}'(u_{k,l}^2)(u^2-u_{k,l}^2) \dd x\right)  + \frac{\alpha}{2}\|u\|^2 + \Iuad + I_{\Delta_k}(u)
			\]
			\State Compute $s = \prox_{\lambda_l \phi_{k,l}}(u_{k,l}-\lambda_l d_{k,l}) -u_{k,l}$ and $b = B_k s$
			\State Set $\alpha_1 = \beta_1^i$, where $i$ is the smallest nonnegative integer such that 
			\[
			f_{k,l+1} + \phi_{k,l}(u_{k,l}+\alpha_1 s) \le f_{k,l} +  \phi_{k,l}(u_{k,l}) + \mu_1 (\alpha (d_{k,l},s) + \phi_{k,l}(u_{k,l}+\alpha_1 s) - \phi_{k,l}(u_{k,l}) )
			\]
			where $f_{k,l+1} = f_{k,l} + \alpha_1 \braket{d_{k,l}, s} + \frac12 \alpha_1^2 \braket{b,s}$, 
			\State Set $u_{k,l+1} = u_{k,l}+\alpha_1 s$ and $d_{k,l+1} = d_{k,l} + \alpha_1 b$
			\State Set $l=l+1$
			\EndWhile
		\end{algorithmic}
	\end{algorithm}
	
	\section{Numerical experiments}\label{sec:numRes}
	We test Algorithm \ref{alg:smoothingTrForLp} with two optimal control problems and compare the results with the proximal gradient method and a majorize-minimization scheme. \\
	In these examples, we work with the quadratic model
	\[
	f_k(u) \coloneqq \braket{g_k, u-u_k} + \frac12 \braket{B_k(u-u_k), u-u_k},
	\]	
	where $g_k = \nabla f(u_k)$ and $B_k = \nabla^2f(u_k)$.
	As described in section \ref{sec:algo}, we first compute a generalized Cauchy point with Algorithm \ref{alg:genCP} and then improve upon it with the SPG algorithm or the MM-SPG algorithm. We also investigate if using NC-GCP improves the performance of the algorithm. Thus, we consider the following methods:
	\begin{itemize}
		\item PG is the proximal gradient method with variable step-size from \cite[section 4.6]{15proxNonsmooth} with a bidirectional line search analogous to Algorithm \ref{alg:genCP}, c.f. \cite{240proxTrustRegion}. We use $\eta = 1e-4$ and $\theta = 0.5$
		to check the condition for sufficient decrease, see \cite[section 6]{15proxNonsmooth}.
		\item MM is the majorize-minimization method from \cite{sparse}.
		\item TR-GCP is the trust region algorithm with only Algorithm \ref{alg:genCP} for the convex GCP case as subproblem solver.
		\item TR-NC-GCP is the trust region algorithm with only Algorithm \ref{alg:genCP} for the nonconvex NC-GCP case as subproblem solver.
		\item TR-SPG is the trust region algorithm with SPG from section \ref{sec:SPG} as second subproblem solver, starting with the result obtained by the convex GCP Algorithm \ref{alg:genCP}.
		\item TR-NC-SPG is the trust region algorithm with SPG from section \ref{sec:SPG} as second subproblem solver, starting with the result obtained by the NC-GCP Algorithm \ref{alg:genCP}.
		\item  TR-MM-SPG is the trust region algorithm with MM-SPG from section \ref{sec:MMSPG} as subproblem solver, starting with the result obtained by the convex GCP Algorithm \ref{alg:genCP}.
		\item  TR-NC-MM-SPG is the trust region algorithm with MM-SPG from section \ref{sec:MMSPG} as subproblem solver, starting with the result obtained by the nonconvex NC-GCP Algorithm \ref{alg:genCP}.
	\end{itemize}
	We remark that we will also test versions NC-SPG and NC-MM-SPG for constrained problems, even though the convergence result in Corollary \ref{cor:gcpYieldsTrialNC} only applies to the unconstrained $\Ltwo$-case. \\
	The following parameters are used in the examples: $\Delta_0 = 10$, $\eta_1 = 10^{-4}$, $\eta_2 = 0.5$, $\gamma_1=\gamma_2 = 0.25$, $\gamma_3=10$, $\mu_1=10^{-4}$, $t_0^A = r_0 = 1$, $\beta_{\text{dec}}=0.5$, $\beta_{\text{inc}} = 10$, $M_{\text{inc}}=2$. If not specified otherwise, we choose $\tau_0=10^{-4}$ for the stopping criterion \eqref{eq:stopCrit} and work on a two-dimensional mesh with $N=256$ nodes in both directions $x$ and $y$. The maxmimum number of iterations in the subproblem solvers SPG and MM-SPG are 10 and 50, respectively. We denote by $K$ the index of the final iteration. \\
	The performance of the different algorithms is compared with respect to the final objective value $F(u_K)$, the number of iterations (iter), the final $\epsilon$-value ($\epsilon_K$), the difference $|u_K-u_{K-1}|$ used as a stopping criterion for the proximal gradient algorithm (dF), the  value $h_K$ and the analog $h_K^{nc}$ with $\phi_k$ replaced by the nonconvex function $\phi_{\epsilon}$, the measure of the domain where $u_K$ vanishes ($|\{u_K=0\}|$), the number of $f$-(feval) and Hessian (Hess) evaluations, the number of nonconvex proximal point evaluations for $\phi_{(\epsilon)}$ (Prox Lp) and the wallclock time ($t$).\\
     We remark that one proximal point computation with the $L^p$-functional is needed to evaluate $h_K^{nc}$, even for the trust-region versions using MM-SPG.\\
	The numerical examples were run on an openSUSE Leap 15.6 Linux machine with a single Intel Core i5-4690 CPU (4 cores, 3.50 GHz base frequency) and 15 GB of RAM. All examples are implemented in Matlab R2024b. \\
	In the SPG subproblem solver and NC-GCP, we have to compute proximal points of $\phi_{\epsilon}$. The resulting minimization problem can not be solved explicitly. We therefore use a gradient method that exploits the structure of the objective function in the proximal point computation: 
	By symmetry, we assume without loss of generality that the argument of the proximal operator is non-negative. If $\epsilon$ is small enough, there is one positive inflection point of the objective when computing $\prox_{r\phi_{\epsilon}}(u)$ that is larger than $\epsilon$. This inflection point is given by 
	\[
	u_0(r) \coloneq \sqrt[^{2-p}]{\frac{rp(1-p)\beta}{1+\alpha r}}.
	\]
	Thus, we can run the gradient method with starting value larger than this inflection point and compare its solution with the minimum on $[0,\epsilon]$ to obtain a proximal point.  This minimum on $[0,\epsilon]$  is given by 
	\[
	M_{\epsilon} \coloneq \min\left(\epsilon, \frac{u}{1+\alpha r + \beta r p/\epsilon^{2-p}}\right).
	\]
	
	We test the algorithm with examples from \cite{lp_cont, 15proxNonsmooth, 240proxTrustRegion}, with $f$ being a tracking type functional given by 
	\[
	f(u) \coloneq \frac12 \|Su-y_d\|^2,
	\]
	where $Su=y_u$ denotes the weak solution of a PDE. Furthermore, we choose $\Uad \coloneq \{ u \in U: -b \le u \le b\}$ for some $b\in\bar\R^+$.\\
	In the implementation we work with finite elements and discretize the controls $u$ with piecewise constants and the state $y_u$ with piecewise linear functions. Then we can solve the PDE numerically to evaluate $F$ and to compute $\cred_k$.\\
	Table \ref{tab:comparisonL1} shows results for a convex example from \cite[section 5.2]{240proxTrustRegion} with $L^1$-regularization. One can observe that in our implementation the trust-region method also performs better than PG just as stated in the paper. However, the difference is not as significant as reported in \cite{240proxTrustRegion}. Therefore, we might expect that also in the nonconvex case the performance differences are not as large.

	\begin{table}[h!]
		\resizebox{\textwidth}{!}{
			\centering
			\begin{tabular}{cccccccccc} 
				\toprule
				alg&$F(u_K)$&iter&dF&$h_K$& $|\{u_K=0\}|$&f eval.&Hess&$t$\\\midrule
				\text{PG}&0.29403&34&1.1102e-16&1.5361e-06&0.41824&110&0&76.5401\\
				\text{TR}&0.29403&6&1.2301e-07&3.8156e-06&0.41815&7&63&24.8788 \\\bottomrule
			\end{tabular}
		}
		\caption{For comparison: results for $p=1$ for the semilinear optimal control problem from \cite[section 5.2]{240proxTrustRegion}.}
		\label{tab:comparisonL1}
	\end{table}
	
	\subsection{Poisson problem}
	First, we consider a linear example from \cite{15proxNonsmooth}, where $y=Su$ is the solution of the Poisson equation 
	\[
	-\Delta y = u \quad \text{ in } \Omega, \qquad y = 0 \quad \text{ on } \partial \Omega .
	\]
	Furthermore, $\Omega \coloneq (0,1)^2$, the desired state $y_d$ is given by
	\[
	y_d(x,y) = 10x \sin(5x)\cos(7y)
	\]
	and $\alpha = 0.01, \beta = 0.01$ and $b=4$.
	Table \ref{tab:poisson} compares the performance of the different schemes for $p=0.9$. One can see that most of the different trust-region variants are faster than PG. However, NC-GCP and NC-MM-SPG are slower. \\
	Table \ref{tab:poisson001} shows similar results for the case $p=0.01$. However, there it is both the MM-SPG variants that are slower than PG.
	Both examples also show that GCP converges to a solution that is less sparse compared to the other computed solutions. 
	
	\begin{table}[h!]
		\resizebox{\textwidth}{!}{
			\centering
			\begin{tabular}{ccccccccccccc}
				\toprule
				alg&$F(u_K)$&iter&$\epsilon_K$&dF&$h_K$&$h_K^{nc}$& $|\{u_K=0\}|$&f eval.&Hess&Prox $L^p$&$t$\\\midrule
				\text{PG}&5.3851&7&0&4.6452e-13&-&3.772e-07&0.52292&20&0&25&47.56\\
				TR-GCP&5.3851&10&2.7557e-18&1.5464e-05&2.3805e-06&0.0016319&0.31263&11&54&1&13.37\\
				TR-NC-GCP&5.3851&7&1.9841e-12&2.2343e-11&1.1305e-06&5.6081e-07&0.52293&8&42&22&37.5212\\
				TR-SPG&5.3851&5&8.3333e-09&1.0844e-09&1.0249e-06&5.0518e-07&0.52251&6&51&13&29.52\\
				TR-NC-SPG&5.3851&6&1.3889e-10&2.3621e-11&2.8675e-07&3.8817e-07&0.52287&7&62&39&68.14\\
				TR-MM-SPG&5.3851&23&3.8682e-47&1.9523e-13&8.9575e-07&4.6945e-07&0.51353&24&1258&1&28.58\\ 
				TR-NC-MM-SPG&5.3851&6&1.3889e-10&2.362e-11&3.2981e-09&4.665e-07&0.52287&7&217&21&73.51\\
				\bottomrule\end{tabular}
		}
		\caption{Poisson problem for $p=0.9$.}
		\label{tab:poisson}
	\end{table}
	
	\begin{table}[h!]
		\resizebox{\textwidth}{!}{
			\centering
			\begin{tabular}{cccccccccccc}
				\toprule
				alg&$F(u_K)$&iter&$\epsilon_K$&dF&$h_K$&$h_K^{nc}$& $|\{u_K=0\}|$&f eval.&Hess&Prox $L^p$&$t$\\\midrule
				\text{PG}&5.3797&6&0&4.9738e-14&0&3.5646e-07&0.55615&16&0&20&30.7835\\
				TR-GCP&5.3839&6&1e-130&0.0001139&1.7089e-06&0.016757&0.089828&7&38&1&9.6772\\
				TR-NC-GCP&5.3797&6&1e-130&6.1298e-10&6.8249e-07&4.6652e-07&0.5484&7&38&21&32.6571\\
				TR-SPG&5.3797&2&1e-130&2.2517e-11&3.9975e-08&3.522e-07&0.54903&3&23&8&13.7695\\
				TR-NC-SPG&5.3797&2&1e-130&1.521e-11&1.8506e-07&2.2061e-07&0.54875&3&23&16&19.2913\\

				TR-MM-SPG&5.3825&6&1e-130&3.858e-15&5.6401e-08&0.0093928&0.15138&7&321&1&69.4626\\ 
				TR-NC-MM-SPG&5.3797&6&1e-130&1.9984e-15&3.4962e-11&3.909e-07&0.54845&7&101&41&57.2521\\
				\midrule
				\text{PG}&5.3789&9&0&2.0428e-14&0&2.2441e-07&0.55608&23&0&30&49.0855\\ 
				TR-GCP&5.3831&6&1e-130&0.00011416&1.3229e-06&0.016875&0.090004&7&38&1&9.8765\\
				\text{TR-NC-GCP}&5.3789&7&1e-130&3.5556e-11&3.898e-07&1.7697e-07&0.5483&8&42&22&39.7132\\
				TR-SPG&5.3789&3&1e-130&9.0697e-12&2.3642e-08&4.1659e-07&0.54851&4&54&31&33.935\\
				TR-NC-SPG&5.3789&2&1e-130&7.8243e-11&3.4016e-07&2.5433e-07&0.5484&3&21&14&15.5141\\
				TR-MM-SPG&5.3816&6&1e-130&1.2098e-07&4.0589e-08&0.0092018&0.15199&7&313&1&75.823\\ 
				TR-NC-MM-SPG&5.3789&7&1e-130&1.1033e-15&4.0281e-11&4.3445e-07&0.54834&8&118&53&82.0246\\
				\bottomrule\end{tabular}
		}
		\caption{Poisson problem for $p=0.01$ for the constrained (upper part) and unconstrained (lower part) case.}
		\label{tab:poisson001}
	\end{table}
	
	\subsection{Semilinear problem} \label{sec:semilinearExample}
	We consider the semilinear PDE
	\[
	-\Delta y + y^3 = u \quad \text{ in } \Omega, \qquad y = 0 \quad \text{ on } \partial \Omega,
	\]
	which was also considered in \cite{15proxNonsmooth} and \cite[section 5.2]{240proxTrustRegion}.
	Here, we investigate the following two examples:
	\begin{equation}\label{S1}
		y_d^1 = 4 \sin(2\pi x)\sin(\pi y)e^x \qquad \text{with} \qquad \alpha =0.002,\, \beta = 0.03 \, \text{ and } \, b=12 \tag{S1}
	\end{equation}
	as in \cite[section 6]{15proxNonsmooth}, and
	\begin{equation}\label{S2}
		y_d^2 = -1 \qquad \text{with} \qquad \alpha = 10^{-4},\, \beta = 10^{-2} \, \text{ and } \, b=25 \tag{S2}
	\end{equation} as in \cite[section 5.2]{240proxTrustRegion}.

	\subsubsection{In \texorpdfstring{$\Ltwo$}{L2}}
	In Tables \ref{tab:semilinear18} and \ref{tab:semilinear28}, one can observe that all the trust-region variants
	outperform the proximal gradient method for Example 2. For Example \eqref{S2}, the trust-region versions need less iterations than PG and are also faster. For Example \eqref{S1}, the performance varies for the trus-region algorithms.
	However, the number of $f$-evaluations is less in the trust-region versions in both cases.\\

	\begin{table}[h!]
		\resizebox{\textwidth}{!}{
			\centering
			\begin{tabular}{cccccccccccc}
				\toprule 
				alg&$F(u_K)$&iter&$\epsilon_K$&dF&$h_K$&$h_K^{nc}$& $|\{u_K=0\}|$&f eval.&Hess&Prox $L^p$&$t$\\\midrule
				\text{PG}&5.8425&7&0&5.8176e-13&0&2.0622e-07&0.35051&19&0&26&58.0639\\
				TR-GCP&5.8488&8&2.4802e-14&0.0015383&8.0169e-06&0.0058667&0.22317&9&46&1&26.2831\\
				TR-NC-GCP&5.8425&7&1.9841e-12&4.0409e-10&2.392e-06&2.6026e-07&0.3511&8&42&22&46.0533\\
				TR-SPG&5.8425&5&8.3333e-09&2.346e-09&1.1119e-07&1.554e-07&0.35194&6&57&19&54.3973\\
				TR-NC-SPG&5.8425&6&1.3889e-10&3.985e-11&4.025e-07&1.0505e-07&0.35089&7&64&41&85.021\\
				TR-MM-SPG&5.8425&7&1.9841e-12&1.4247e-11&1.3016e-06&1.8651e-07&0.34167&8&349&1&142.4404\\ 
				TR-NC-MM-SPG&5.8425&7&1.9841e-12&3.5111e-13&1.3681e-08&1.7429e-07&0.35108&8&216&31&131.0837\\
				\bottomrule\end{tabular}
		}
		\caption{Semilinear example \eqref{S1} for $p=0.8$.}
		\label{tab:semilinear18}
	\end{table}

	\begin{table}[h!]
		\resizebox{\textwidth}{!}{
			\centering
			\begin{tabular}{cccccccccccc}
				\toprule 
				alg&$F(u_K)$&iter&$\epsilon_K$&dF&$h_K$&$h_K^{nc}$& $|\{u_K=0\}|$&f eval.&Hess&Prox $L^p$&$t$\\\midrule
				\text{PG}&0.23947&196&0&6.3838e-16&-&7.9564e-07&0.39395&442&0&615&948.9267\\
				TR-GCP&0.24111&35&9.6776e-77&1.3076e-05&3.6169e-06&0.0010277&0.39099&36&184&1&99.2605\\
				TR-NC-GCP&0.23933&27&9.1837e-57&2.2965e-08&3.351e-06&1.2092e-06&0.41539&28&163&83&165.1155\\
				TR-SPG&0.23945&13&1.6059e-24&2.2497e-11&2.6166e-08&1.6686e-07&0.40265&14&260&232&326.5575\\
				TR-NC-SPG&0.23929&9&2.7557e-16&3.4039e-10&5.1926e-07&6.8518e-08&0.42947&10&163&195&245.5838\\
				TR-MM-SPG&0.23939&7&1.9841e-12&4.8579e-11&8.0666e-07&1.595e-07&0.41182&8&394&1&167.5447\\ 
				TR-NC-MM-SPG&0.23932&7&1.9841e-12&2.0327e-09&1.6862e-06&1.2026e-07&0.43298&8&312&24&158.726\\\midrule
				\text{PG}&0.23805&111&0&4.1828e-13&-&1.5228e-06&0.43752&297&0&371&598.0928\\
				TR-GCP&0.23914&101&1e-130&5.4706e-07&3.184e-05&0.00031692&0.47694&101&569&1&321.6377\\
				TR-NC-GCP&0.23793&63&1e-130&3.9061e-09&4.3679e-06&3.6087e-06&0.45364&64&350&162&371.1875\\
				TR-SPG&0.23802&19&8.2206e-38&1.1316e-10&1.8947e-07&4.542e-07&0.44208&20&386&313&483.4491\\
				TR-NC-SPG&0.23791&7&1.9841e-12&2.2096e-09&2.0729e-06&2.4572e-07&0.45755&8&169&195&254.0695\\
				TR-MM-SPG&0.23798&14&1.1471e-26&3.2161e-11&7.679e-07&3.2199e-07&0.44817&15&766&1&324.2317\\ 
				TR-NC-MM-SPG&0.23789&7&1.9841e-12&1.9871e-11&3.7959e-07&2.9924e-07&0.47508&8&371&29&185.6843\\
				\bottomrule\end{tabular}
		}
		\caption{Semilinear example \eqref{S2} for $p=0.8$ for the constrained (upper part) and unconstrained (lower part) case.}
		\label{tab:semilinear28}
	\end{table}
	
	Table \ref{tab:semilinear2-001} displays results for a small value $p=0.01$. One can observe that the results obtained with the different methods are not as similar as in the other cases. Especially the versions only using proximal points of the convex upper bound do not lead to very sparse solutions. This might be explained by the fact that the term $\psi_{\epsilon_k}'(u_k^2)$ appearing in the linearization in $\phi_k$ is either almost zero or very large. We also note that for such small $p$, we also choose very small values for the sequence $(\epsilon_k)$ because $\epsilon$ is raised to the power $p$ in the smooth approximation $\psi_{\epsilon_k}$.
	
	Furthermore, Table \ref{tab:semilinear2-03tau} compares results for different values of the parameter $\tau_0$ from the stopping criterion \eqref{eq:stopCrit}. In most cases, not a lot more iterations are needed to obtain higher accuracy.
	
	\begin{table}[h!]
		\resizebox{\textwidth}{!}{
			\centering
			\begin{tabular}{cccccccccccc}
				\toprule 
				alg&$F(u_K)$&iter&$\epsilon_K$&dF&$h_K$&$h_K^{nc}$& $|\{u_K=0\}|$&f eval.&Hess&Prox $L^p$&$t$\\\midrule
				\text{PG}&0.15699&185&0&9.8527e-13&-&2.5776e-07&0.12666&385&0&560&723.029\\
				TR-GCP&0.15814&31&1e-130&1.071e-06&4.0806e-06&0.00020802&0.045013&32&176&1&99.1437\\
				TR-NC-GCP&0.15696&24&1e-130&1.0415e-08&3.9807e-06&3.7055e-06&0.14946&25&140&69&137.9696\\
				TR-SPG&0.15718&4&1e-130&7.3917e-08&1.0911e-07&1.6915e-07&0.18857&5&75&99&103.1187\\
				TR-NC-SPG&0.15718&4&1e-130&7.0686e-08&1.0802e-07&1.6196e-07&0.1886&5&76&117&117.5869\\
				TR-MM-SPG&0.1575&4&1e-130&1.8382e-06&2.0585e-06&3.3232e-07&0.088486&5&94&1&42.1471\\ 
				TR-NC-MM-SPG&0.15688&6&1e-130&3.0315e-11&1.7708e-10&3.0491e-07&0.14258&7&118&28&72.9381\\
				
				\bottomrule\end{tabular}
		}
		\caption{Semilinear example \eqref{S2} for $p=0.01$. 
		}
		\label{tab:semilinear2-001}
	\end{table}
	
	\begin{table}[h!]
		\resizebox{\textwidth}{!}{
			\centering
			\begin{tabular}{cccccccccccccc} 
				\toprule
				$\tau_0$&alg&$F(u_K)$&iter&$\epsilon_K$&dF&$h_K$&$h_K^{nc}$& $|\{u_K=0\}|$&f eval.&Hess&Prox $L^p$&$t$\\\midrule
				1e-4&TR-SPG&0.17005&4&0.0001&1.0595e-07&4.5525e-07&1.536e-07&0.2421&5&83&104&111.2823\\
				&TR-MM-SPG&0.17044&4&0.0001&1.728e-07&3.8581e-06&2.3202e-07&0.24167&5&148&1&57.2107\\\midrule
				1e-6&TR-SPG&0.17004&5&1e-05&9.3771e-10&4.1042e-08&1.9816e-07&0.2421&6&98&111&119.4242\\
				&TR-MM-SPG&0.17042&6&1e-06&1.8942e-11&2.066e-10&2.1286e-07&0.24167&7&199&1&73.3469\\\midrule
				1e-8&TR-SPG&0.17003&16&1e-15&4.9405e-15&1.4958e-09&1.7923e-07&0.2421&16&273&223&238.8856\\
				&TR-MM-SPG&0.17042&6&1e-06&1.8942e-11&2.066e-10&2.1286e-07&0.24167&7&199&1&76.2117\\\bottomrule
			\end{tabular}
		}
		\caption{Semilinear example \eqref{S2} for $p=0.3$ for different values of $\tau_0$.}
		\label{tab:semilinear2-03tau}
	\end{table}

	\subsubsection*{Localized control}
	Next, we consider an example with a localized control, so we consider the modified semilinear PDE
	\[
	-\Delta y + y^3 = u\chi_{\omega} \quad \text{ in } \Omega, \qquad y = 0 \quad \text{ on } \partial \Omega,
	\]
	for some measurable subset $\omega \subseteq \Omega$. Here we choose $\omega = B_{0.4}( (0.6, 0.4))$. Table \ref{tab:localized} shows results for Example \ref{S2}.  One can observe that the trust-region versions perform more efficiently than PG.
	Table \ref{tab:localizedMeshSizes} presents results for different mesh sizes. The wallclock time increases for smaller mesh sizes as to be expected, but the number of iterations stays the same. This indicates the trust-region algorithm is mesh independent. 
	
	\begin{table}[h!]
		\resizebox{\textwidth}{!}{
			\centering
			\begin{tabular}{cccccccccccc}
				\toprule 
				alg&$F(u_K)$&iter&$\epsilon_K$&dF&$h_K$&$h_K^{nc}$& $|\{u_K=0\}|$&f eval.&Hess&Prox $L^p$&$t$\\\midrule
				\text{PG}&0.21949&65&0&8.9836e-13&-&2.3757e-07&0.51598&139&0&201&179.4807\\
				TR-GCP&0.22205&15&7.6472e-29&2.144e-05&3.1954e-06&0.00052933&0.5459&16&74&1&41.3081\\
				TR-NC-GCP&0.21939&22&8.8968e-45&5.6677e-09&3.2348e-06&1.8996e-06&0.5291&23&125&60&85.7864\\ 
				TR-SPG&0.21952&5&8.3333e-09&2.2829e-11&1.7172e-08&1.0129e-07&0.53806&6&94&116&75.9283\\
				TR-NC-SPG&0.2194&7&1.9841e-12&2.299e-11&7.5034e-08&1.8749e-08&0.53397&8&99&130&89.636\\
				TR-MM-SPG&0.21942&5&8.3333e-09&2.8599e-12&1.1969e-08&1.0299e-07&0.53301&6&179&1&79.2657\\ 
				TR-NC-MM-SPG&0.21939&6&1.3889e-10&9.2577e-10&2.6294e-07&8.1947e-08&0.53107&7&155&21&74.7195\\\midrule
				\text{PG}&0.20844&37&0&3.8858e-16&-&5.8858e-07&0.60659&130&0&147&153.5447\\
				TR-GCP&0.21133&101&1e-130&4.7092e-07&9.0465e-06&0.00023579&0.63374&101&545&1&298.209\\  
				\text{TR-NC-GCP}&0.20841&33&1.1516e-71&2.2118e-08&4.0511e-06&4.2746e-06&0.61291&34&181&83&128.7169\\
				TR-SPG&0.20849&7&1.9841e-12&1.5953e-10&7.9971e-08&3.5677e-07&0.60822&8&155&151&116.8268\\
				TR-NC-SPG&0.20836&5&8.3333e-09&2.4859e-08&1.5097e-06&4.6399e-07&0.62275&6&105&132&88.3064\\
				TR-MM-SPG&0.20844&5&8.3333e-09&3.0631e-11&1.7156e-07&3.708e-07&0.60759&6&286&1&118.1968\\ 
				TR-NC-MM-SPG&0.20841&6&1.3889e-10&3.6295e-12&4.1405e-08&3.6728e-07&0.60967&7&243&30&108.6518\\
				\bottomrule\end{tabular} 
		}
		\caption{Semilinear example \eqref{S2} with localized control for $p=0.5$ for the constrained (upper part) and unconstrained (lower part) case.}
		\label{tab:localized}
	\end{table}
	
	\begin{table}[h!]
		\resizebox{\textwidth}{!}{
			\centering
			\begin{tabular}{cccccccccccccc}
				\toprule
				$N$&alg&$F(u_K)$&iter&$\epsilon_K$&dF&$h_K$&$h_K^{nc}$& $|\{u_K=0\}|$&f eval.&Hess&Prox $L^p$&$t$\\\midrule
				128&TR-SPG&0.23992&4&0.0001&1.8287e-07&4.1572e-06&1.6687e-06&0.55502&5&69&96&28.0599\\
				&TR-MM-SPG&0.2399&5&1e-05&7.2945e-09&3.5619e-08&1.3257e-07&0.5564&6&200&1&18.4843\\\midrule
				256&TR-SPG&0.23994&4&0.0001&1.8255e-07&4.1819e-06&1.671e-06&0.55524&5&69&96&118.8294\\
				&TR-MM-SPG&0.23992&5&1e-05&7.2421e-09&4.4985e-08&1.3369e-07&0.55618&6&212&1&95.2375\\\midrule
				512&TR-SPG&0.23994&4&0.0001&1.8216e-07&4.199e-06&1.6684e-06&0.55513&5&69&96&601.2989\\
				&TR-MM-SPG&0.23992&5&1e-05&7.2258e-09&2.7043e-08&1.3425e-07&0.55612&6&224&1&519.2016\\
				\bottomrule
			\end{tabular} 
		}
		\caption{Semilinear example \ref{S2} with localized control and $p=0.7$ for different mesh sizes.}
		\label{tab:localizedMeshSizes}
	\end{table}
	
	\subsubsection*{In \texorpdfstring{$\Hone$}{H01}}
	If $V=\Hone$, we use MM-SPG to solve the trust-region subproblems in order to avoid computing proximal points of the $L^p$-functional in $\Hone$.
	Recall that MM-SPG only requires evaluations of the proximal map for the convex upper bounds of the objective. We compare the
	algorithm with the majorize minimization scheme (MM) developed in \cite{sparse}. In both examples in Tables \ref{tab:H1S1} and \ref{tab:H1S2}, MM is slower than TR-MM-SPG and the trust-region algorithm needs considerably less $f$-evaluations.
	
	\begin{table}[h!]
		\resizebox{\textwidth}{!}{
			\centering
			\begin{tabular}{cccccccccc}
				\toprule 
				alg&$F(u_K)$&iter&$\epsilon_K$&dF&$h_K$&$|\{u_K=0\}|$&f eval.&t\\\midrule
				\text{MM}&6.2142&151&1.232e-08&1.4112e-09&8.4623e-11&0.37673&303&325.8849\\
				TR-MM-SPG&6.2142&5&8.3333e-09&9.6167e-07&7.8252e-07&0.37495&13&196.4098\\
				\bottomrule\end{tabular}
		}
		\caption{Semilinear example S1 in $\Hone$ with $p=0.5$.}\label{tab:H1S1}
	\end{table}
	
	\begin{table}[h!]
		\resizebox{\textwidth}{!}{
			\centering
			\begin{tabular}{ccccccccc}
				\toprule
				alg&$F(u_K)$&iter&$\epsilon_K$&dF&$h_K$&$|\{u_K=0\}|$&f eval.&t\\\midrule
				MM&0.49201&251&3.2724e-13&9.1407e-12&1.0715e-13&0.45484&520&533.244\\ 
				MM-SPG&0.49208&8&1e-09&2.7061e-05&8.7889e-06&0.49984&19&290.1151\\
				\bottomrule\end{tabular}
		}
		\caption{Semilinear example S2 in $\Hone$ with $p=0.9$.}\label{tab:H1S2}
	\end{table}
	
	\subsection{Remarks on NC-GCP}
	When comparing the versions with only Algorithm \ref{alg:genCP} as subproblem solver, TR-NC-GCP mostly requires fewer iterations than TR-GCP across all examples, although it is slower overall. TR-NC-GCP leads in most cases also to results that are closer to the other numerical solutions. \\
	When comparing the results of TR-SPG with TR-NC-SPG and TR-MM-SPG with TR-NC-MM-SPG, the performance seems to be depending on the example.
	However, one can observe that among the unconstrained examples, the nonconvex versions mostly yield better results. Therefore, the fact that Lemma \ref{lm:proxDistInequ} and the subsequent Corollary \ref{cor:gcpYieldsTrialNC} only hold in the unconstrained case might also be relevant in practical implementations. \\
	Furthermore, choosing TR-NC-MM-SPG over TR-MM-SPG or TR-NC-GCP over TR-GCP helps for small values of $p$, when TR-MM-SPG does not really lead to sparse solutions. Tables \ref{tab:poisson001} and \ref{tab:semilinear2-001} show that TR-NC-MM-SPG leads to results that are sparser and closer to those obtained with the other algorithms.

	\bibliographystyle{plainnat} 
	\bibliography{trustRegionRestructure.bib}
	
\end{document}